\PassOptionsToPackage{dvipsnames}{xcolor} 

\documentclass[11pt,letterpaper]{amsart}
\usepackage[a4paper, left=1in, right=1in, top=1.1in, bottom=1.1in]{geometry}

\usepackage[english]{babel}
\usepackage[utf8]{inputenc}
\usepackage{amsmath}
\usepackage{graphicx}
\usepackage{amssymb}
\usepackage{stmaryrd}
\usepackage{amsthm}
\usepackage{tikz-cd}
\usepackage{tikz}
\usetikzlibrary{arrows.meta,calc}
\usepackage{mathrsfs}
\usepackage{csquotes}
\usepackage[colorinlistoftodos]{todonotes}
\usepackage{yfonts}
\usepackage{ dsfont }
\usepackage{xcolor}
\usepackage{multicol}
\usepackage{indentfirst}
\usepackage[hyphens]{url}
\usepackage{hyperref}
\usepackage[hyphenbreaks]{breakurl}
\usepackage{subfiles}
\usepackage{longtable}
\usepackage{float}
\usepackage{subcaption}
\usepackage{mathabx}

\definecolor{tocblue}{RGB}{0,51,153}
\colorlet{tocblueLight}{tocblue!90!white} 
\definecolor{mrred}{HTML}{B33333} 
 \usepackage[
    backend=biber,
    style=alphabetic, maxbibnames=99, url=false
  ]{biblatex}
\usepackage{xurl}
\hypersetup{
    colorlinks,
    linkcolor={mrred},
    citecolor={tocblueLight},
    urlcolor={RoyalPurple}
}
\hypersetup{breaklinks=true}

\usepackage[inline]{enumitem}   
\makeatletter

\newcommand{\inlineitem}[1][]{%
\ifnum\enit@type=\tw@
    {\descriptionlabel{#1}}
  \hspace{\labelsep}%
\else
  \ifnum\enit@type=\z@
       \refstepcounter{\@listctr}\fi
    \quad\@itemlabel\hspace{\labelsep}%
\fi}
\makeatother
\allowdisplaybreaks

 \usepackage[
    backend=biber,
    style=alphabetic, maxbibnames=99, url=false
  ]{biblatex}
\usepackage{cleveref}

\title[Geometrically Abelian Sections on Higher-Dimensional Varieties]{Geometrically Abelian Sections on\\ Higher-Dimensional Varieties}

\author{Daebeom Choi}
\address{Department of Mathematics\\
    University of Pennsylvania\\
    Philadelphia, PA 19104-6395}
\email{dbchoi@sas.upenn.edu}

\date{\today}
\theoremstyle{definition}
\newtheorem{thm}{Theorem}[section]

\newtheorem{lem}[thm]{Lemma}

\newtheorem{prop}[thm]{Proposition}

\newtheorem{defn}[thm]{Definition}

\newtheorem{defthm}[thm]{Definition-Theorem}

\newtheorem{cor}[thm]{Corollary}

\newtheorem{rmk}[thm]{Remark}

\crefname{thm}{Theorem}{Theorems}
\crefname{lem}{Lemma}{Lemmas}
\crefname{prob}{Problem}{Problems}
\crefname{prop}{Proposition}{Propositions}
\crefname{qes}{Question}{Questions}
\crefname{defn}{Definition}{Definitions}
\crefname{eg}{Example}{Examples}
\crefname{ex}{Exercise}{Exercises}
\crefname{conj}{Conjecture}{Conjectures}
\crefname{defthm}{Definition-Theorem}{Definition-Theorems}
\crefname{cor}{Corollary}{Corollaries}
\crefname{claim}{Claim}{Claims}
\crefname{rmk}{Remark}{Remarks}
\crefname{section}{Section}{Sections}

\newcommand{\R}{\mathbb{R}}

\newcommand{\Z}{\mathbb{Z}}
\newcommand{\N}{\mathbb{N}}
\newcommand{\Q}{\mathbb{Q}}

\newcommand{\Pic}{\text{Pic}}
\newcommand{\HH}{\text{H}}

\newcommand{\T}{\mathbb{T}}

\newcommand{\Br}{\mathrm{Br}}

\newcommand{\Hom}{\text{Hom}}

\newcommand{\Alb}{\text{Alb}}
\newcommand{\ab}{\text{ab}}

\DeclareSymbolFont{yhlargesymbols}{OMX}{yhex}{m}{n}
\DeclareMathAccent{\widetriangle}{\mathord}{yhlargesymbols}{"E6} 

\begin{document}
    
\maketitle

\begin{abstract}

We investigate how much arithmetic information about a variety can be detected by sections of the geometrically abelian fundamental groups of sufficiently small open subvarieties. Over a \(p\)-adic field, such a section determines a rational point on the Albanese torsor, and, under a conjectural assumption, this point is represented by a \(0\)-cycle of degree \(1\) on the original variety if the variety has index \(1\). For curves, the latter statement holds unconditionally. Over finitely generated fields, such sections still force the relative Brauer group to vanish away from the characteristic. The key input is an identification of the Chern class obstruction attached to a section with the obstruction arising from local Tate duality. As a byproduct, we give a new proof of an open-subvariety strengthening of a theorem of Esnault and Wittenberg.

\end{abstract}

\section{Introduction}\label{sec:intro}

Grothendieck's section conjecture, originating in \cite{Gro97b,Gro97} as part of the broader anabelian program, predicts that for a hyperbolic curve \(X\) over a finitely generated field \(K\), every section of \(\pi_1(X)\to G_K\) arises from a \(K\)-rational point of \(X\). Subsequent work extended aspects of these anabelian phenomena, including versions of the section conjecture, to other classes of fields, notably \(p\)-adic fields; see, for example, \cite{Moc99,BT08,Pop25}. For surveys on the section conjecture, we refer to \cite{Pop11,Sai11,Sti13b}.

In light of the recent claimed proof of the \(p\)-adic section conjecture by OpenAI \cite{OAI26}, it is natural to ask whether analogous phenomena persist in higher dimensions, as suggested already by Grothendieck's original formulation for higher-dimensional varieties and sufficiently small open subvarieties \cite{Gro97b}. A complementary motivation is that anabelian phenomena can often be detected in remarkably small quotients of the fundamental group \cite{Moc99, BT08, Pop10c, Pop17, ST22, Pop25}, though this is not yet known for the full section conjecture even for curves; we therefore investigate how much arithmetic information is encoded by sections of the geometrically abelian fundamental group of a sufficiently small open subvariety.

Our approach builds on the Chern class obstructions introduced by Sa\"idi \cite{Sai11, Sai12}, whose vanishing is forced by geometrically abelian sections over sufficiently small open subvarieties. Over a \(p\)-adic field, we identify these obstructions with local Tate duality pairings, revealing a natural hierarchy: the obstructions associated with \(\Pic^0\) characterize the geometricity of the induced Albanese section (\cref{thm:main 2}), while those arising from the full Picard group govern the finer question of representing the resulting Albanese point by a \(0\)-cycle of degree \(1\), conditionally for varieties of index \(1\) (\cref{thm:albanese coker}) and unconditionally for curves (\cref{cor:albanes coker}). Their vanishing also yields period \(1\) and the vanishing of the relative Brauer group (\cref{thm:main 2,thm:newmain}).

Before stating our main results, we make the ``sufficiently small open subset'' condition precise. We introduce the following conditions, where \(\widehat{(-)}\) denotes profinite completion:
\[\text{(VP0)}\quad \widehat{\Pic^0(U)}=0,\qquad \text{(VP1)}\quad \widehat{\Pic(U)}=0.\]
By the (topological) finite generation of Picard groups, every variety \(X\) over a \(p\)-adic or a finitely generated field admits an open subset \(U\subseteq X\) satisfying these conditions.

\begin{thm}\label{thm:main}
    Let \(K\) be a \(p\)-adic field, \(A\) be an abelian variety over \(K\), and \(U\subseteq A\) be an open subvariety satisfying \textup{(VP0)}. Then a section \(s:G_K\to\pi_1(A)=\pi_1(A)^{(\ab)}\) is \textit{geometric} if and only if it lifts to a section \(\tilde{s}:G_K\to\pi_1(U)^{(\ab)}\). In particular, every section \(\tilde{s}:G_K\to\pi_1(U)^{(\ab)}\) determines a \(K\)-rational point of \(A\).
\end{thm}

Here, \(\pi_1(X)^{(\text{ab})}\) denotes the \textit{geometrically abelian etale fundamental group}, i.e. the extension
\[ 1\to \pi_1(X_{\bar{K}})^{\text{ab}} \to \pi_1(X)^{(\text{ab})}\to G_K\to 1. \] 
The case of elliptic curves in \cref{thm:main} is essentially established in \cite{Sai21}.

\Cref{thm:main} is a consequence of a more intrinsic statement about Chern class obstructions and the Albanese torsor. Recall that the Chern class obstructions are the pullbacks \(s^\ast c_1(L)\in\HH^2(K,\mu_n)\) of Chern classes of line bundles along a section; they can also be defined using geometrically abelian sections; see \cref{prop:F1 comparison,defn:Chern class} and the discussion therein. These are obstructions to a section being geometric because \( s^\ast c_1(L)=0\) for every geometric section, since every line bundle is trivial on a point. Following Sa\"idi, we say that a section is \textit{orthogonal to \(\Pic^0\)/\(\Pic\)} (cf. \cref{defn:good section}) if all of its Chern class obstructions vanish, and \textit{uniformly orthogonal to \(\Pic^0\)/\(\Pic\)} if the same holds after every finite separable extension.

For any smooth variety \(X\), there exists a semi-abelian variety \(\Alb_X^0\), a torsor \(\Alb_X^1\) over \(\Alb_X^0\), and a map \(a:X\to \Alb_X^1\) universal among such pairs (cf. \cref{defthm:albanese}). The period of \(X\) is the order of the element \([\Alb_X^1]\in \HH^1(K, \Alb_X^0) \). Note that if we have a section \(s:G_K\to \pi_1(X)^{(\ab)}\), then by composing with \(a_\ast\) we obtain a section \(s^\text{a}:G_K\to \pi_1(\Alb_X^1)\).

\begin{thm}\label{thm:main 2}
    Let \(X\) be a smooth proper variety over a \(p\)-adic field \(K\), and let \(s:G_K\to \pi_1(X)^{(\ab)}\) be a section. Then the following are equivalent:
    \begin{enumerate}
        \item \(s\) is orthogonal to \(\Pic^0\).
        \item \(s\) is uniformly orthogonal to \(\Pic^0\).
        \item The induced section \(s^\text{a}:G_K\to \pi_1(\Alb_X^1)\) of the Albanese torsor is geometric.
    \end{enumerate}
\end{thm}

\begin{cor}
    If \(X\) admits a section orthogonal to \(\Pic^0\), then the period of \(X\) is \(1\).
\end{cor}

The case of curves is already contained in \cite{Sai21}.

The main step in the proof of \cref{thm:main 2} is to identify the Chern class obstruction with the obstruction arising from local Tate duality. Let \(X\) be a smooth proper variety over a \(p\)-adic field \(K\), and let \(s_1,s_2:G_K\to \pi_1(X)^{(\ab)}\) be two geometrically abelian sections. Local Tate duality gives a perfect pairing
\[
e:\HH^1(K,\pi_1(X_{\overline K})^{\ab}\otimes \Z/n)\times \HH^1(K,\Pic_X[n])\longrightarrow \Br(K)[n].
\]
Let \(L\) be a line bundle on \(X\) such that \(c_1(L_{\overline K})\in \HH^2(X_{\overline K},\mu_n)\) vanishes. We show that the pullbacks \(s_i^\ast c_1(L)\in \Br(K)[n]\), \(i=1,2\), are well defined and satisfy
\[
s_1^\ast c_1(L)-s_2^\ast c_1(L)=e(\delta L\cup [s_1-s_2]).
\]
This comparison is established in \cref{sec:weil}, in particular in \cref{thm:Chern vs Weil}. Related computations of \(s^\ast c_1(L)\) already appear in the literature; see, for example, \cite{EW09}. The point of our formulation is that it identifies the Chern class obstruction directly with the pairing appearing in local Tate duality. In addition, the argument does not require \(X\) to be a \(K(\pi,1)\) and depends only on the geometrically abelian section, two features that are crucial for the applications below. The structure of the argument of the paper can be summarized as follows. The conditions \textup{(VP0)} and \textup{(VP1)} force the vanishing of the relevant Chern class obstructions (cf. \cref{lem:albanese lifting}). By \cref{thm:Chern vs Weil}, these obstructions are identified with the corresponding Weil pairings. Pairing against \(\Pic^0\) controls the geometricity of the induced section on \(\Alb_X^1\), as in \cref{thm:main 2}, while pairing against the N\'eron--Severi group conjecturally detects whether the resulting Albanese point is represented by a \(0\)-cycle, as in \cref{thm:albanese coker}, and also controls the vanishing of the relative Brauer group, as in \cref{thm:newmain}.

We now describe three consequences of this picture, concerning the index and the relative Brauer group, \(0\)-cycles of degree \(1\), and the generic period. The first consequence concerns the index. While an analogue of \cref{thm:main 2} for the index appears to be substantially harder in general, the same Chern class obstructions control the relative Brauer group, and hence the discrepancy between the period and the index for curves.

\begin{thm}\label{thm:newmain}
    Let \(K\) be a \(p\)-adic field or a finitely generated field, and let \(X\) be a smooth proper variety over \(K\). If there exists an open subvariety \(U\) satisfying \textup{(VP1)} and a section \(s_U:G_K\to\pi_1(U)^{(\ab)}\), then
    \[\Br(X/K)\text{ is }\begin{cases}
        \text{ trivial} & \text{if }\operatorname{char}K=0,\\
        l\text{-primary} & \text{if }\operatorname{char}K=l.
    \end{cases}\]
    In particular, if \(X\) is a curve, then the same holds for the ratio of the index to the period.
\end{thm}

A rational version of this statement was proved in \cite{EW10}; see also \cite{Wit08}, whereas our proof is independent of these results.

The second consequence concerns whether the Albanese point produced by \cref{thm:main 2} comes from a \(0\)-cycle of degree \(1\) on \(X\). Assume that \(X\) has index \(1\). By \cref{thm:main 2}, a section orthogonal to \(\Pic^0\) determines a rational point on \(\Alb_X^1\), and it is natural to ask whether this point is represented by a \(0\)-cycle of degree \(1\). This amounts to understanding the cokernel of the Albanese map \(\operatorname{CH}_0(X)^0\to\Alb_X^0(K)\). Over a \(p\)-adic field, Kai \cite{Kai16} proposed the following description of the cokernel of the Albanese map as a higher-dimensional analogue of Lichtenbaum duality for curves:
\[\textup{(K)}\ \operatorname{coker}\left(\operatorname{CH}_0(X)^0\to\Alb_X^0(K)\right)\text{ is Pontryagin dual to }\operatorname{coker}\left(\Pic(X_{\overline{K}})^{G_K}\to\operatorname{NS}(X_{\overline{K}})^{G_K}\right),\]
and proved it in several special cases.

\begin{thm}\label{thm:albanese coker}
Let \(X\) be a smooth projective variety of index \(1\) over a \(p\)-adic field. Then \textup{(K)} is equivalent to the following statement: the points of \(\Alb_X^1(K)\) represented by \(0\)-cycles of degree \(1\) are precisely those arising from sections of \(\pi_1(X)^{(\ab)}\to G_K\) that are orthogonal to \(\Pic\).
\end{thm}

Since Kai's conjecture is known for curves and in the setting of \cite[Theorem 1.1]{Kai16}, the following is immediate from \cref{thm:albanese coker}, \cref{cor:period index} and \cref{lem:albanese lifting,lem:Picard_vaishing}.

\begin{cor}\label{cor:albanes coker}
    Let \(X\) be a smooth proper variety over a \(p\)-adic field \(K\). Assume that either \textup{(1)} \(X\) is a curve, or \textup{(2)} \(X\) has a smooth projective model over \(\mathcal{O}_K\) with a smooth Picard scheme and has index \(1\). Let \(U\subseteq X\) be an open subset satisfying \textup{(VP1)}. If there exists a section \(s_U:G_K\to\pi_1(U)^{(\ab)}\), then \(X\) has index \(1\), and the induced point of \(\Alb_X^1\) is represented by a \(0\)-cycle of degree \(1\).
\end{cor}

Finally, the third consequence relates geometrically abelian sections to the generic period. Recall that the generic period of \(X/K\) is the supremum of the periods of its nonempty open subsets. A theorem of Esnault-Wittenberg and Wittenberg characterizes generic period \(1\) in terms of an abelianized birational section. We obtain an integral refinement of this characterization, in which the birational section is replaced by a section over a suitable open subset. Moreover, our proof is independent of the arguments in \cite{EW10,Wit08}.

\begin{thm}\label{thm:main3}
    Let \(X\) be a smooth proper variety over a \(p\)-adic field \(K\). Fix a finite Galois extension \(F/K\) such that \(X(F)\ne\emptyset\) and \(\text{NS}(X_F)=\text{NS}(X_{\widebar{K}})\), and choose an open subset \(U\subseteq X\) such that \(U_F\) satisfies \textup{(VP1)}. Then there exists a section of \(\pi_1(U)^{(\ab)}\to G_K\) if and only if the generic period of \(X\) is \(1\). If \(X\) is a curve, this is equivalent to \(X\) having index \(1\).
\end{thm}

More broadly, these results suggest a systematic approach to studying the arithmetic information encoded by sections of fundamental groups. One may vary both the open subvariety and the quotient of its geometric fundamental group, and ask which arithmetic obstructions become detectable at each stage. The present paper investigates this question at the geometrically abelian level, where Chern classes and arithmetic duality provide a common framework for understanding Albanese torsors, relative Brauer groups, and zero-cycles. A natural direction for future research is to extend this framework to geometrically nilpotent quotients and investigate what further arithmetic information they can detect. We regard the results of this paper as a first step toward this broader problem.

Finally, we describe the organization of the paper. In \cref{sec:prelim}, we review basic results on the Albanese torsor and the Albanese variety of a smooth variety. \Cref{sec:weil} contains the technical core of the paper, where we prove that the obstructions arising from local Tate duality coincide with the Chern class obstructions. Using this, we prove \cref{thm:main 2} in \cref{sec:section}. In \cref{sec:brauer}, we study the relation between Chern class obstructions and the relative Brauer group and prove \cref{thm:newmain}. \cref{sec:proof} proves \cref{thm:main,thm:main3} using the relation between liftability of a section and Chern class obstructions established in \cref{lem:albanese lifting}. Finally, in \cref{sec:albanese}, we relate the remaining obstruction to the Albanese cokernel and Kai's paper \cite{Kai16}.

\subsection{Notations and Conventions}\label{subsec:notation}

In this paper, a variety is always assumed to be geometrically integral, and \(\widebar{K}\) always denotes the separable closure of \(K\). For a \(K\)-scheme \(X\), \(\pi_1(X)\) denotes the étale fundamental group, \(\Pic_X\) and \(\Pic_X^0\) denote the Picard scheme and its identity component, and \(\Pic(X)\) and \(\Pic^0(X)\) denote the Picard group and the subgroup consisting of classes contained in \(\Pic_X^0\). Note that \( \Pic^0(X)\subseteq \Pic_X^0(K) \) need not be an equality. For a line bundle \(L\), we denote its Chern class with \(\mu_n\)-coefficients in \(\HH^2_{\acute et}(X,\mu_n)\) by \(c_{1,n}(L)\). When there is no risk of confusion, we simply write \(c_1(L)\). We write \(G_K\) for the absolute Galois group. Recall that \(\pi_1(X)^{(\text{ab})}\) is the geometrically abelian étale fundamental group, i.e. the extension
\[ 1\to \pi_1(X_{\bar{K}})^{\text{ab}} \to \pi_1(X)^{(\text{ab})}\to G_K\to 1, \]
and \(\pi_1(X)^{(\Z/n)} \) for the geometrically \(n\)-torsion abelian étale fundamental group, i.e. the extension
\[ 1\to \pi_1(X_{\bar{K}})^{\text{ab}}/n \to\pi_1(X)^{(\Z/n)} \to G_K\to 1. \]
For an abelian variety \(A\), \(A^t\) denotes the dual abelian variety. Finally, \(\HH^i\) denotes profinite group cohomology, and \(\HH^i_{\acute et}\) denotes étale cohomology. For a field, we will denote both étale cohomology and group cohomology by \( \HH^i(K, \mathcal{F})\).

\section*{Acknowledgements}
The author would like to thank his advisor, Florian Pop, for his continued support and many helpful discussions. He is also grateful to Danny Krashen for his help with the proof of \cref{prop:chern class}, and to Jakob Stix for his valuable comments on a preliminary version of this paper.

\section{Albanese variety and torsor}\label{sec:prelim}

Here we summarize several well-known results on Albanese varieties, Albanese torsors, and their relation to the Picard scheme. We also present \cref{lem:albanese line bundle 1} and \cref{lem:albanese line bundle 2}, which will be used in the proof of \cref{lem:albanese lifting}.

\begin{defthm}\label{defthm:albanese}(\cite{Ser58, Ser88, DS08, Wit08})
    Let \(X\) be a smooth variety over a field \(K\). Then there exist a semi-abelian variety \(\Alb_X^0\), an \(\Alb_X^0\)-torsor \(\Alb_X^1\), and a map \(a:X\to \Alb_X^1\) that is universal among such data. In other words, if \(A\) is a semi-abelian variety, \(P\) is an \(A\)-torsor, and there is a map \(X\to P\), then there exists a unique map \(\Alb_X^1\to P\) making the following diagram commute:
    \[
\begin{tikzcd}[ampersand replacement=\&]
    X\arrow[r, "a"] \arrow[rd] \&  
    \Alb_X^1\arrow[d] \\
    \& P
\end{tikzcd}.
\]
\(\Alb_X^0\) and \(\Alb_X^1\) are respectively called \textbf{Albanese variety} and \textbf{Albanese torsor} of \(X\). The order of \( [\Alb_X^1]\in \HH^1(K, \Alb_X^0) \) is called the \textbf{period} of \(X\). If \(X\) is also proper, then \(\Alb_X^0\) is an abelian variety.
\end{defthm}

Note that Albanese varieties and Albanese torsors exist for more general classes of varieties. Here, for simplicity, we restrict our discussion to smooth varieties. We refer to \cite{Wit08, Moc12} for surveys on Albanese varieties. We now summarize the facts that will be useful for us.

\begin{thm}\label{thm:albanese}
    Let \(X\) be a smooth proper variety, and let \(U\subseteq X\) be an open subset.

    \begin{enumerate}
        \item (\cite[Corollary A.5]{Wit08}) The Albanese variety and Albanese torsor are stable under base change by a separable extension.
        
        \item (\cite[Theorem A.4]{Wit08}, \cite[Remark 2.1]{ACV26}) \(\Alb_X^0\) is the dual of \( \Pic_{X, \text{red}}^0 \). In particular, \(a:X\to \Alb_X^1\) induces pullback isomorphisms
        \[ a_1^\ast : \Pic^0_{\Alb_X^1}\xrightarrow{\sim} \Pic^0_{X, \text{red}} \text{ and } a_2^\ast : \Pic^0(\Alb_X^1)\xrightarrow{\sim} \Pic^0(X). \]
        
        \item (\cite[Theorem A.4]{Wit08}) Let \(\text{Div}^0_{X\setminus U}(X_{\widebar{K}})\) be the subgroup of \( \text{Div}_{X\setminus U}(X_{\widebar{K}}) \), i.e. divisors of \(X_{\widebar{K}}\) supported on \(X\setminus U\), consisting of algebraically trivial divisors, equipped with the natural \(G_K\)-action. Let \(T\) be a torus over \(K\) with character lattice \(\text{Div}^0_{X\setminus U}(X_{\widebar{K}})\). Then \(T\) is the kernel of \(\Alb_U^0\to\Alb_X^0\), and \(\Alb_U^1\to\Alb_X^1\) is a \(T\)-torsor over \(\Alb_X^1\).
        
        \item (\cite[Proposition 4.4]{MT03}) The map \(a:U\to \Alb_U^1\) induces a surjection
        \[ \pi_1(U)^{(\text{ab})} \twoheadrightarrow \pi_1(\Alb_U^1)^{(\text{ab})}, \]
        and it is an isomorphism if \(X=U\) and the N\'eron-Severi group of \(X\) is torsion-free. If \(n\) is invertible in \(K\) and the divisors supported on \(X_{\widebar{K}}\setminus U_{\widebar{K}}\) generate the N\'eron-Severi group of \(X_{\widebar{K}}\), then
        \[ \pi_1(U)^{(\Z/n)} \simeq \pi_1(\Alb_U^1)^{(\Z/n)}, \]
        where \(\pi_1^{(\Z/n)}\) denotes the largest quotient whose geometric part is an abelian \(n\)-torsion profinite group.
    \end{enumerate}
\end{thm}

\begin{rmk}
    In \cref{thm:albanese} (2), we write \(a_1^\ast\) and \(a_2^\ast\) separately to distinguish the two pullbacks, since both are used in this section. In all other sections, however, \(a^\ast\) denotes \(a_2^\ast\).
\end{rmk}

We include the following verification only to clarify the distinction between the scheme-theoretic and group-theoretic pullbacks \(a_1^*\) and \(a_2^*\); readers willing to use \cref{thm:albanese} may skip directly to \cref{lem:albanese line bundle 1}.

(1) and (3) are stated in this form in the reference. We give some further details on (2) and (4). For (4), \cite[Proposition 4.4]{MT03} proves the statement under the assumption that \(K\) is algebraically closed. Moreover, by \cite[Theorem B]{ACV26}, the Albanese variety is invariant up to purely inseparable isogeny under inseparable base extension, so the same statement also holds over a separably closed field. Note that the N\'eron-Severi group over \(\widebar{K}\) is the same as the N\'eron-Severi group over an algebraic closure, since it is defined as the group of connected components of \(\Pic\), and algebraic extensions of separably closed fields do not change connected components. Hence the condition is also invariant, so (4) holds over a separably closed field as well. The general statement then follows, since the Galois part is unchanged and it remains only to compare the geometric part.

The first assertion of (2) is also stated in this form in the reference. We explain the two maps \(a_1^\ast, a_2^\ast\) in (2) in some detail, since they will be used in the proof of \cref{lem:albanese line bundle 1} and for the convenience of the reader. Note that in (2), \(a_1^\ast\) is a morphism of schemes, whereas \(a_2^\ast\) is a morphism of abelian groups.

To prove that \(a_1^\ast : \Pic^0_{\Alb_X^1}\to\Pic^0_{X, \text{red}}\) is an isomorphism, we may assume that \(X\) has a \(K\)-rational point \(x\), by Galois descent and \cref{thm:albanese} (1). Here, we directly prove that \(\Alb_X^1\) is the dual abelian variety of \(\Pic^0_{X, \text{red}}\). Note that \(\Pic^0_{X, \text{red}}\) is an abelian variety by \cite[Corollary 3.2]{Gro95}. By the universal property, there exists a Poincaré line bundle \(P\) on \(X\times \Pic^0_{X} \) such that \(P|_{x\times \Pic^0_{X}}\) is trivial and \(P|_{X\times y}\) is the line bundle corresponding to \(y\in \Pic^0_{X}\). Pulling back, this induces a line bundle \(P_{\text{red}}\) on \( X\times \Pic^0_{X, \text{red}} \). Since \(P_{\text{red}}|_{x\times \Pic^0_{X, \text{red}}}\) is trivial, this induces a morphism \( a:X\to \Pic^{0,t}_{X, \text{red}} \).

We prove that this satisfies the condition in \cref{defthm:albanese}, and hence that \(\Pic^{0,t}_{X, \text{red}}\) is the Albanese torsor and Albanese variety. Since \(X\) is a proper pointed variety, we may assume that the target is an abelian variety \(A\). Let \(f:X\to A\) be the given map. This induces a pullback map \(f^\ast:\Pic^0_A\to \Pic^0_X \), which factors through \( \Pic^0_A\to \Pic^0_{X, \text{red}} \). By taking the dual of this map, we obtain \(\Pic^{0,t}_{X, \text{red}}\to A\), and this satisfies the condition in \cref{defthm:albanese} by the identification with the double dual. More precisely, we have the following diagram

\[
\begin{tikzcd}[ampersand replacement=\&]
   P_X\arrow[d]\& (1, f^\ast)^\ast P_X=(f,1)^\ast P_A  \arrow[d] \arrow[r] \arrow[l]\& P_A \arrow[d]\\
    X\times \Pic^0_{X, \text{red}}\& X\times A^t \arrow[r, "f\times 1"]\arrow[l,"1\times f^\ast"'] \&  A\times A^t
\end{tikzcd}.
\]

This proves that when \(X(K)\ne \emptyset\), we have \(\Alb_X^1\simeq\Pic^{0,t}_{X, \text{red}}\). As mentioned above, this implies that \(a_1^\ast\) is an isomorphism. For a smooth proper variety \(Y\), the inclusion \( \Pic^0(Y)\subseteq \Pic^0_Y(K) \) is characterized as the kernel of the Brauer obstruction (cf. \cite{Sti10}). Since the Brauer obstruction is functorial, it follows that \(a_2^\ast\) is also an isomorphism.

Note that \cref{thm:albanese} (2) implies
\[ \Pic_{X, \text{red}}^0 \simeq  \Pic^0_{\Alb_X^1}\simeq \Pic^0_{\Alb_X^0} \]
since the last isomorphism holds for an arbitrary abelian variety and its torsor (see, for example, \cite[Remark 2.1]{ACV26}). By taking \(K\)-points, we obtain identifications
\begin{equation}\label{eqn:Pics}
    \Pic^0(X)=\Pic^0(\Alb_X^1)\subseteq \Pic^0_{\Alb_X^0}(K) =\Pic^0(\Alb_X^0)
\end{equation}
where the first equality is via \(a_2^\ast\), and the second equality holds since \(\Alb_X^0(K)\ne \emptyset\). 

With \cref{eqn:Pics} in mind, we prove two lemmas that will be used in the proof of \cref{lem:albanese lifting}. Let \(X\) be a smooth proper variety, let \(U\subseteq X\) be an open subset, and let \(T_U\) be the torus corresponding to the lattice \(\text{Div}^0_{X\setminus U}(X_{\widebar{K}})\), i.e. \(T_U:=\ker (\Alb_U^0\to \Alb_X^0)\).

\begin{lem}\label{lem:albanese line bundle 1}
    Choose a divisor \(E\) on \(X\) supported on \(X\setminus U\) whose class is contained in \(\Pic^0\), and let \(\chi:T_U\to \mathbb{G}_m \) be the corresponding character. Let \(\Alb_{U, \chi}^i\) be the pushforward of the \(T_U\)-torsor \(\Alb_U^i\to \Alb_X^i\) along \(\chi\) for \(i=0,1\), i.e.
        \[
        \begin{tikzcd}[ampersand replacement=\&]
            \Alb_U^i \arrow[r, "\chi_\ast"] \arrow[d] \& \Alb_{U, \chi}^i \arrow[ld, "\pi"]\\
            \Alb_X^i 
        \end{tikzcd}.
        \]
        Then, for \(i=0,1\), the \(\mathbb{G}_m\)-torsors  \(\Alb_{U, \chi}^i\to \Alb_X^i\) correspond, via the identification in \cref{eqn:Pics},  to the line bundle \(\mathcal{O}_X(E)\).
\end{lem}

\begin{proof}
    For \(i=0\), this follows directly from \cite[Theorem A.4]{Wit08}. The case of \(i=1\) is then a consequence of the identification \cref{eqn:Pics}.
\end{proof}

Next, we prove \cref{lem:albanese line bundle 2}, which plays a crucial role in the proof of \cref{lem:albanese lifting} (2). Consider open subsets \(U\subseteq V\subseteq X\), and assume that the divisors supported on \(X_{\widebar{K}}\setminus V_{\widebar{K}}\) generate the N\'eron-Severi group of \(X_{\widebar{K}}\). Under this assumption, we analyze the map \(\Alb_U^0\to \Alb_V^0\) (the first part of this is also done in \cite[Lemma 2.4]{Wit08}). Note that these are extensions of \(\Alb_X^0\) by tori corresponding to \(\text{Div}^0_{X\setminus U}(X_{\widebar{K}})\) and \(\text{Div}^0_{X\setminus V}(X_{\widebar{K}})\). Since
\[ \text{Div}^0_{X\setminus V}(X_{\widebar{K}})\hookrightarrow \text{Div}^0_{X\setminus U}(X_{\widebar{K}}) \]
and the cokernel is \(\text{Div}_{V\setminus U}(V_{\widebar{K}})\), which is a permutation \(G_K\)-module, we obtain that the kernel \(T\) of \(\Alb_U^0\to \Alb_V^0\) is a quasi-split torus of the form \( \prod_{j=1}^{m}\text{Res}_{F_j/K}\mathbb{G}_m \), where \(F_j\) is the maximal separable subextension of the field of constants of the irreducible component \(D^j\) of \(V\setminus U\).  Let \( \chi_j:T_{F_j}\to \mathbb{G}_m\) be the character of \(T_{F_j}\) corresponding to the factor \(\text{Res}_{F_j/K}\mathbb{G}_m\) and the natural multiplication map \( (F_j\otimes_K F_j)^\times \to F_j^\times\). The character \(\chi_j\) corresponds to \( D^j_{\widebar{K}} \in \text{Div}_{V\setminus U}(V_{\widebar{K}})\). Note that \(D^j_{\widebar{K}}\) is neither irreducible nor reduced in general.

\begin{lem}\label{lem:albanese line bundle 2}
    In the above setting, let \(U\subseteq V\subseteq X\), and assume that the divisors supported on \(X_{\widebar{K}}\setminus V_{\widebar{K}}\) generate the N\'eron-Severi group of \(X_{\widebar{K}}\). Let \(\chi_j\) be as above, and let \(\Alb_{U, F_j, \chi_j}^i\) be the pushforward of the \(T\)-torsor \(\Alb_{U, F_j}^i\to \Alb_{V, F_j}^i\) along \(\chi_j\) for \(i=0,1\), i.e.
        \[
        \begin{tikzcd}[ampersand replacement=\&]
            \Alb_{U, F_j}^i\arrow[r, "\chi_{j,\ast}"] \arrow[d] \& \Alb_{U, F_j, \chi_j}^i \arrow[ld, "\pi"]\\
            \Alb_{V, F_j}^i
        \end{tikzcd}.
        \]
    Then \(\Alb_{U, F_j, \chi_j}^i\) are \(\mathbb{G}_m\)-bundles over \(\Alb_{V, F_j}^i\) such that some positive tensor power is the pullback of an algebraically trivial line bundle on \( \Alb_{X,F_j}^i\).  
\end{lem}

\begin{proof}
    Let \(T_U\) and \(T_V\) be the kernels of \(\Alb_U^0\to \Alb_X^0\) and \(\Alb_V^0\to \Alb_X^0\), respectively. Then we have an exact sequence
    \[ 1\to T\to T_U\to T_V\to 1 \]
    whose character lattices fit into an exact sequence
    \[ 0 \to \text{Div}^0_{X\setminus V}(X_{\widebar{K}})\hookrightarrow \text{Div}^0_{X\setminus U}(X_{\widebar{K}}) \twoheadrightarrow \text{Div}_{V\setminus U}(V_{\widebar{K}})\to 0.  \]
    The character \(\chi_j\) corresponds to \(D^j_{F_j}\), and is \(G_{F_j}\)-invariant. Therefore, by taking Galois cohomology with respect to \(G_{F_j}\), there exists a positive integer \(n\) such that \(nD^j_{F_j}\) lifts to a \(G_{F_j}\)-invariant element of \(\text{Div}^0_{X\setminus U}(X_{\widebar{K}})\), i.e. \(\chi_j^{n}\) lifts to a character \(\psi_j\) of \(T_{U, F_j}\). Then \(\Alb_{U, F_j, \chi_j}^{i, \otimes n}\) is the pushforward of the \(T_{U, F_j}\)-torsor \(\Alb_{U, F_j}^i\to \Alb_{X, F_j}^i\) along \(\psi_j\). Since \(\psi_j\in \text{Div}^0_{X\setminus U}(X_{\widebar{K}})^{G_{F_j}}=\text{Div}^0_{X\setminus U}(X_{F_j})\), this corresponds to an algebraically trivial line bundle on \(\Alb_{X,F_j}^i\) by \cref{lem:albanese line bundle 1}.  
\end{proof}

\section{Chern classes vs. the Weil Pairing}\label{sec:weil}

The goal of this section is to prove \cref{thm:Chern vs Weil}, which identifies the Chern class obstruction with the obstruction arising from the Weil pairing for general varieties. We begin with the case of abelian varieties, both for motivation and because it is used later in the paper.

Fix a natural number \(n\). Let \(K\) be a field of characteristic prime to \(n\), and let \(A\) be an abelian variety over \(K\). Let \(O\) be the identity of \(A\), let \(s_O:G_K\to \pi_1(A)\) be the corresponding section of the fundamental exact sequence
\[ 1\to \pi_1(A_{\widebar{K}})\to \pi_1(A)\to G_K\to 1, \]
and let \(s:G_K\to \pi_1(A)\) be another section. Passing to the maximal quotient whose geometric part is \(n\)-torsion, we obtain
\[ 1\to \pi_1(A_{\widebar{K}})\otimes \Z/n\simeq A[n]\to \pi_1(A)^{[n]}\to G_K\to 1, \]
hence the difference \(s-s_O\) modulo \(n\) defines a class
\[ [s-s_O]\in \HH^1(K, A[n]). \]
We define the following map:
\[ f_{s,n}:A^t(K)\xrightarrow{\delta} \HH^1(K, A^t[n])\xrightarrow{e(-\cup [s-s_O])}\HH^2(K, \mu_n), \]
where \(e\) is the Weil pairing, and \(\delta\) is the boundary map associated to the Kummer exact sequence for \(A^t\).

Note that \(A^t(K)\) is canonically identified with \(\Pic^0(A)\), the group of degree \(0\) line bundles on \(A\). Let
\[  A^t(K)\to \HH^1_{\acute et}(A, \mathbb{G}_m) \]
be the map sending a \(K\)-point of \(A^t\) to the corresponding line bundle on \(A\). By composing this with the boundary map associated to the Kummer exact sequence for \(\mathbb{G}_m\),
\[ c_1:\HH_{\acute et}^1(A, \mathbb{G}_m)\to \HH_{\acute et}^2(A, \mu_n)\simeq \HH^2(\pi_1(A), \mu_n), \]
and the pullback map
\[ s^\ast:\HH^2(\pi_1(A), \mu_n)\to \HH^2(K, \mu_n), \]
we obtain a morphism
\[ g_{s, n}: A^t(K)\to \HH^2(K, \mu_n). \]
Note that \(\HH_{\acute et}^2(A, \mu_n)\simeq \HH^2(\pi_1(A), \mu_n)\) follows from the general theory of \textit{algebraic \(K(\pi, 1)\)-spaces}. However, later we will define \(g_{s, n}\) directly, without appealing to this; see \cref{prop:F1 comparison} and \cref{defn:Chern class}.

The following is the most important case of the main theorem of this section, \cref{thm:Chern vs Weil}.

\begin{prop}\label{prop:coin}
    \( f_{s,n}=g_{s,n}. \)
\end{prop}

The difficulty of the proof arises from the following discrepancy: in the definition of \(g_{s,n}\), the map \(c_1\) involves the étale cohomology of \(\mathbb{G}_m\), which is a purely geometric, étale-cohomological object and cannot be reduced to group cohomology, since \(\mathbb{G}_m\) is not locally constant. However, the morphism \(s^\ast\) uses \(s\), which is merely a group-theoretic section and may not come from geometry. Hence, we need to reconcile these two aspects in order to compute \(g_{s,n}\).

What we will do is the following. Let \(L\) be an algebraically trivial line bundle on \(A\), corresponding to a point of \(A^t(K)\). Since \(L\) is algebraically trivial, its Chern class \(c_1(L)\in \HH^2_{\acute et}(A, \mu_n)\) is geometrically trivial, i.e. its image under
\[ \HH^2_{\acute et}(A, \mu_n)\to \HH^2_{\acute et}(A_{\widebar{K}}, \mu_n) \]
is trivial. The Leray spectral sequence for \(A\to \text{Spec }K\) gives
\[  E_2^{p,q}=\HH_{\acute et}^p(K, \HH_{\acute et}^q(A_{\widebar{K}}, \mu_n))\Rightarrow \HH_{\acute et}^p(A, \mu_n). \]
Since \(A\to \text{Spec }K\) admits a section, we have \(E_2^{p,0}=\HH_{\acute et}^p(K, \mu_n)\hookrightarrow \HH_{\acute et}^p(A,\mu_n)\). Therefore, \(E_2^{p,0}=E_\infty^{p,0}\) and \(d_r^{p,1}=0\) for \(r\ge 2\). In particular, \(E_2^{2,0}=E_\infty^{2,0}\) and \(E_2^{1,1}=E_\infty^{1,1}\). We record this as a simple lemma, since it will be used several times.
\begin{lem}\label{lem:SS}
    Let \(E_2^{p,q} \Rightarrow H^{p+q}\) be a spectral sequence such that \(E_2^{p, 0}\hookrightarrow H^p\). Then \(E_2^{p,0}=E_\infty^{p,0}\), and \(d_r^{p,1}=0\) for \(r\ge 2\). In particular, \(E_2^{0,1}=E_\infty^{0,1}\), \(E_2^{2,0}=E_\infty^{2,0}\) and \(E_2^{1,1}=E_\infty^{1,1}\).
\end{lem}

By definition,
\[ F^1\HH^2_{\acute et}(A, \mu_n):=\ker \left(\HH^2_{\acute et}(A, \mu_n)\to \HH^2_{\acute et}(A_{\widebar{K}}, \mu_n)\right),  \]
and by \cref{lem:SS}, we have an exact sequence
\begin{equation}\label{eqn:etexact}
    0\to \HH_{\acute et}^2(K, \mu_n)\to  F^1\HH^2_{\acute et}(A, \mu_n)\xrightarrow{\eta} \HH_{\acute et}^1(K, \HH_{\acute et}^1(A_{\widebar{K}}, \mu_n))\to 0.
\end{equation}

The proof of \cref{prop:coin} proceeds as follows. As a first step, we compute \(\eta(c_1(L))\). For this, we use the description of classes in \(\HH^2_{\acute et}(A, \mu_n)\) in terms of \(\mu_n\)\textit{-gerbes} over \(A\). In particular, \(c_1(L)\) is described by the \(n\)th \textit{root gerbe} (or \textit{root stack}) of \(L\). Since \(c_1(L)\) is geometrically trivial, we can describe \(\eta(c_1(L))\) as the \textit{Galois descent obstruction} of this root gerbe, and compute it in terms of an explicit cocycle in group cohomology. Consequently, \(\eta(c_1(L))\) is precisely the Galois descent obstruction to an \(n\)th root of \(L\).

This reduces everything to group cohomology. The second step is to compute \(g_{s,n}\) using explicit cocycles in group cohomology. We have an analogue of \cref{eqn:etexact}, namely
\begin{equation}\label{eqn:grpexact}
    0\to \HH^2(K, \mu_n)\to  F^1\HH^2(\pi_1(A), \mu_n)\xrightarrow{\eta} \HH^1(K, \HH^1(\pi_1(A_{\widebar{K}}), \mu_n))\to 0,
\end{equation}
coming from the Hochschild-Serre spectral sequence, by \cref{lem:SS}. The class \(c_1(L)\) lies in \(F^1\HH^2(\pi_1(A), \mu_n)\), and we know \(\eta(c_1(L))\). This is enough to determine \(c_1(L)\). Indeed, \cref{eqn:grpexact} admits a splitting \(F^1\HH^2(\pi_1(A), \mu_n)\to \HH^2(K, \mu_n)\) given by 
\[ F^1\HH^2(\pi_1(A), \mu_n)\hookrightarrow \HH^2(\pi_1(A), \mu_n)\xrightarrow{s_O^\ast}\HH^2(K, \mu_n). \]
Moreover, since \(s_O\) is geometric,
\[ s_O^\ast c_1(L)=c_1(s_O^\ast(L))=0 \]
because every line bundle on \(\text{Spec }K\) is trivial. From \(\eta(c_1(L))\) and \(s_O^\ast c_1(L)\) we can determine \(c_1(L)\) in terms of group cohomology, which allows us to compute \(g_{s,n}=s^\ast c_1(L)\).\

Let us consider the first step. We prove the following result.

\begin{prop}\label{prop:chern class}
    Let \(X\) be a proper variety over a field \(K\) such that \((n, \text{char }K)=1 \), and let \(L\) be a line bundle such that \(M^{\otimes n}\simeq L|_{X_{\widebar{K}}} \) for some \(M\in \Pic (X_{\widebar{K}})\). Then \(c_1(L)\in F^1\HH^2(X, \mu_n)\), and
    \[ \eta(c_1(L))\in \HH^1_{\acute et}(K, \HH^1_{\acute et}(X_{\widebar{K}}, \mu_n))=\HH^1_{\acute et}(K, \Pic(X_{\widebar{K}})[n] ) \]
    is represented by the cocycle
    \[ \sigma\in G_K\mapsto M^\sigma\otimes M^{-1}\in  \HH^1(X_{\widebar{K}}, \mu_n)=\Pic(X_{\widebar{K}})[n]. \]
\end{prop}

We compute \(\eta\) using root gerbes. Let \(X\) be a connected scheme and \(\mathcal{G}\) an étale sheaf of abelian groups on \(X\). There is a canonical bijection between isomorphism classes of étale \(\mathcal{G}\)-gerbes on \(X\) and cohomology classes in \(\HH_{\acute et}^2(X,\mathcal{G})\); see \cite[Définition IV.3.1.1, Théorème IV.3.4.2]{Gir71}. In our case, we have a line bundle \(L\) on \(X\) and its Chern class \(c_1(L)\in \HH_{\acute et}^2(X, \mu_n)\). The corresponding \(\mu_n\)-gerbe, namely the \(n\)th \textbf{root gerbe} \(\sqrt[n]{L}\), is given by the functor \(\sqrt[n]{L}:\textbf{Sch}_X^{\acute et}\to \textbf{Groupoid}\) defined by
\[ \sqrt[n]{L}(f:T\to X)= \{(M, \phi) \ |\   M\in \Pic(T), \phi:M^{\otimes n}\xrightarrow{\sim}f^\ast L  \} \]
where a morphism between \( (M_i, \phi_i) \), \(i=1,2\), is given by an isomorphism \(\psi:M_1\to M_2\) such that \( \phi_1=\phi_2\circ \psi^{\otimes n} \). See, e.g. \cite[proof of Proposition 2.3.4.4]{Lie07}. This description suggests that \(\eta(c_1(L))\) should be the descent obstruction to an \(n\)th root of \(L\). This interpretation is natural and may be familiar to experts, but we include the argument for completeness.

\begin{lem}\label{lem:giraud} (\cite[Exercise V.3.1.9.2]{Gir71} for the etale site of schemes)
    Let \(f:X\to Y\) be a morphism of schemes, and let \(\mathcal{G}\) be an etale sheaf of abelian groups on \(X\). Let \(\mathscr{G}\) be a \(\mathcal{G}\)-gerbe whose class lies in \(F^1\HH^2_{\acute et}(X, \mathcal{G})\). Let \(\mathcal{F}\) be the sheaf of maximal sub-gerbes of the direct image \(f_\ast \mathscr{G}\). Then \(\mathcal{F}\) is an \(R^1f_\ast \mathcal{G}\)-torsor on \(Y\), and its class \([\mathcal{F}]\in \HH^1_{\acute et}(Y, R^1f_\ast \mathcal{G})\) is the image of \([\mathscr{G}]\) under the map 
    \[ \eta:F^1\HH^2_{\acute et}(X, \mathcal{G})\to \HH^1_{\acute et}(Y, R^1f_\ast \mathcal{G}) \]
    arising from the Leray spectral sequence for \(f\).
\end{lem}

Now we apply this to \(\mathcal{G}=\mu_n\) and \(\mathscr{G}=\sqrt[n]{L}\). Then \(f_\ast\mathscr{G}\) is given by
\[
(f_\ast\mathscr{G})(K')=\{(M,\phi)\mid M\in\Pic(X_{K'}),\ \phi:M^{\otimes n}\xrightarrow{\sim}L|_{X_{K'}}\}
\]
for every finite separable extension \(K'/K\). Recall that any two objects of a gerbe are locally isomorphic. We claim that two pairs \((M_1,\phi_1)\) and \((M_2,\phi_2)\) are locally isomorphic if and only if \(M_1=M_2\) in \(\Pic(X_{K'})\). Indeed, if they become isomorphic after a finite separable extension \(K''/K'\), then \(M_1=M_2\), since the natural map \(\Pic(X_{K'})\hookrightarrow\Pic(X_{K''})\) is injective, as \(X\) is proper. Conversely, suppose that \(M_1=M_2\) in \(\Pic(X_{K'})\). After choosing an identification \(M_1\simeq M_2\), the two isomorphisms \(\phi_1\) and \(\phi_2\) differ by multiplication by an element of \(K'^\times\). After a finite separable extension adjoining an \(n\)th root of this element, the two pairs become isomorphic.

Therefore, a subgerbe of \(f_\ast\mathscr{G}\) is determined by the choice of an \(n\)th root \(M\) of \(L\) in \(\Pic(X_{\widebar K})\). The corresponding subgerbe \(\mathscr{G}_0\) is given by
\[
\mathscr{G}_0(K')=\{(M',\phi')\mid M'|_{X_{\widebar K}}=M\}.
\]
This is a \(\mu_n\)-gerbe: for every object \((M',\phi')\), its automorphisms are precisely multiplication by \(n\)th roots of unity on \(M'\). Hence, the maximal subgerbes of \(f_\ast\mathscr{G}\) correspond to choices of such \(M\), and \(\mathcal{F}\) is the sheaf assigning to \(K'\) the set of \(n\)th roots of \(L|_{X_{K'}}\) defined over \(K'\). Since \(X\) is proper,
\[
\HH^1_{\acute et}(X_{\widebar K},\mu_n)\simeq\Pic(X_{\widebar K})[n],
\]
and \(\Pic(X_{\widebar K})[n]\) acts on the set of \(n\)th roots by tensor product, \(M\mapsto M\otimes N\). Thus, \(\mathcal{F}\) is a \(\Pic(X_{\widebar K})[n]\)-torsor.

Finally, we compute the cohomology class of \(\mathcal{F}\). It is the descent obstruction of any choice of \(M\), and is therefore represented by the cocycle
\[
\sigma\longmapsto M^\sigma\otimes M^{-1}.
\]
This completes the proof of \cref{prop:chern class}.

Now consider the second step. This step is a purely algebraic group cohomology computation, so for clarity we work with abstract profinite groups instead of étale fundamental groups. Fix a positive integer \(n\) and let
\begin{itemize}
    \item \(G, N\) be profinite groups such that \(N^{\text{ab}}/n\) is finitely generated, 
    \item \(G\) acts on \(N\), and let \(\Pi:=N\rtimes G\) be the corresponding semidirect product,
    \item \(s_O:G\to \Pi\) be the canonical section corresponding to the semidirect product, and \(s:G\to \Pi\) be another section,
    \item \(W\) be a discrete \(n\)-torsion \(\Pi\)-module, on which \(N\) acts trivially.
\end{itemize}
Then, since there exists a section, we obtain
\[ 0\to \HH^2(G, W)\to F^1\HH^2(\Pi, W)\xrightarrow{\eta} \HH^1(G, \HH^1(N,W))=\HH^1(G, \Hom(N,W))\to 0, \]
by \cref{lem:SS}

The maps \(s_O^*\) and \(s^*\) give two maps \( F^1\HH^2(\Pi, W)\to \HH^2(G,W) \). Since they agree on \(\HH^2(G, W)\), \(s^\ast-s_O^\ast\) factors through \(\HH^1(G, \Hom(N,W))\). Moreover, as before, \(s\) and \(s_O\) are two sections, so after taking the abelianization of \(N\) and reducing mod \(n\), we obtain
\[ 1\to N^{\text{ab}}/n\to \overline{\Pi}\to G\to 1 \]
and two sections of this exact sequence. The difference between these two sections defines a cohomology class \([s-s_O]\in \HH^1(G,N^{\text{ab}}/n)\).

Note that there is a natural bilinear pairing of \(G\)-modules
\[ e: \Hom(N, W)\times \left(N^{\text{ab}}\otimes \Z/n\right)\to W \]
since \(W\) is \(n\)-torsion. We will prove the following.

\begin{prop}\label{prop:group coh calc}
    The map
    \[ s^\ast -s_O^\ast : \HH^1(G, \Hom(N,W))\to \HH^2(G,W) \]
    is given by \( e(-\cup [s-s_O]  ) \), i.e. by taking the cup product with \([s-s_O]\) via the pairing \(e\).
\end{prop}

\begin{proof}
    The main idea of the proof is to describe explicitly the decomposition of \(F^1\HH^2(\Pi, W)\) as a direct sum of \(\HH^2(G, W)\) and \(\HH^1(G, \Hom(N,W))\) induced by \(s_O\). Let \(\xi\) be the cohomology class in \(\HH^1(\Pi, N^{\text{ab}}/n)\) represented by the cocycle
    \[ (n, g)\in \Pi=N\rtimes G\mapsto [n]\in N^{\text{ab}}/n. \]
    It is straightforward to check that this is a cocycle. Moreover, by construction, we have
    \[ s_O^\ast \xi=0, \quad s^\ast \xi=[s-s_O], \quad \text{and }\quad \text{res}^N_{\Pi}(\xi)=\pi_n \]
    where \(\pi_n\) is the quotient map \(N\to N^\text{ab}/n \). Now, for any \(c\in F^1\HH^2(\Pi, W)\), we claim that
    \begin{equation}\label{eqn:dec}
        c=\text{inf}^{G}_{\Pi}(s_O^\ast c)+e( \text{inf}^{G}_{\Pi}\eta(c) \cup \xi). 
    \end{equation}
    First, we check that the right hand side lies in \(F^1\HH^2(\Pi, W)\). The class \(\text{inf}^{G}_{\Pi}(s_O^\ast c)\) clearly lies in \(F^1\HH^2(\Pi, W)\). The class \(e( \text{inf}^{G}_{\Pi}\eta(c) \cup \xi)\) also lies in \(F^1\HH^2(\Pi, W)\), since \(\text{inf}^{G}_{\Pi}\eta(c)\) lies in \(F^1\HH^1(\Pi, \text{Hom}(N, W))\) and the Hochschild--Serre spectral sequence is multiplicative.
    
    Hence, to prove \cref{eqn:dec}, it suffices to show that the two sides have the same images under \(s_O^\ast\) and \(\eta\). We have
    \[ s_O^\ast e( \text{inf}^{G}_{\Pi}\eta(c) \cup \xi)=e( s_O^\ast\text{inf}^{G}_{\Pi}\eta(c) \cup s_O^\ast\xi)=0 \]
    since \(s_O^\ast \xi=0\). Therefore, both sides of \cref{eqn:dec} have the same image under \(s_O^\ast\).
    
    We next apply \(\eta\). Since \(\text{inf}^{G}_{\Pi}(s_O^\ast c)\in F^2\HH^2(\Pi, W)\), it lies in the kernel of \(\eta\). Moreover, by the functoriality of the Hochschild--Serre spectral sequence,
    \[ \eta e( \text{inf}^{G}_{\Pi}\eta(c) \cup \xi)=e( \eta( \text{inf}^{G}_{\Pi}\eta(c) \cup \xi)). \]
    We now use the multiplicativity of the Hochschild--Serre spectral sequence for the \(\Pi\)-modules \(\text{Hom}(N,W)\) and \( N^{\text{ab}}/n \). The class \(\text{inf}^{G}_{\Pi}\eta(c)\) lies in \(F^1\HH^1(\Pi, \text{Hom}(N, W))\), and its \(E^{1,0}\)-component is \(\eta(c)\). On the other hand, the \(E^{0,1}\)-component of \(\xi\) is \(\text{res}^N_{\Pi}(\xi)=\pi_n\). Therefore, \(\eta( \text{inf}^{G}_{\Pi}\eta(c) \cup \xi)\), which is the \(E^{1,1}\)-component of \(\text{inf}^{G}_{\Pi}\eta(c) \cup \xi\), is \(\eta(c)\cup \pi_n\). Hence,
    \[  \eta e( \text{inf}^{G}_{\Pi}\eta(c) \cup \xi)=e(\eta(c)\cup \pi_n)=\eta(c) \]
    since evaluation against \(\pi_n\) is the identity on \(\Hom(N,W)\). This proves \cref{eqn:dec}. Applying \(s^\ast-s_O^\ast\) to \cref{eqn:dec} and using \(s^\ast\xi=[s-s_O]\) and \(s_O^\ast\xi=0\), we obtain
    \[ (s^\ast-s_O^\ast)(c)=e(\eta(c)\cup[s-s_O]) \]
    as desired.
\end{proof}

We now prove \cref{prop:coin}. Indeed, we will prove a generalization of it to an arbitrary smooth proper variety, and while doing so, we will introduce a notion, \cref{defn:Chern class}, that will be used later.

\begin{prop}\label{prop:F1 comparison}
    Let \(X\) be a connected scheme over \(K\) and \(\mathcal{F}\) be a locally constant abelian sheaf. Then
    \[  \HH^2(\pi_1(X), \mathcal{F})\hookrightarrow  \HH^2_{\acute et}(X, \mathcal{F}).\]
    This map induces an isomorphism
    \[ F^1\HH^2(\pi_1(X), \mathcal{F})\xrightarrow{\sim}  F^1\HH^2_{\acute et}(X, \mathcal{F})  \]
    where \(F^1\) is defined, respectively, by the Hochschild-Serre spectral sequence for \( \pi_1(X)\to G_K \) and by the Leray spectral sequence for \(X\to \text{Spec }K\). Moreover, if the monodromy of \(\mathcal{F}\) factors through \(G_K\), then
\[F^1\HH^2(\pi_1(X)^{(\mathrm{ab})},\mathcal{F}) \xrightarrow{\sim}F^1\HH^2(\pi_1(X),\mathcal{F}) \xrightarrow{\sim} F^1\HH^2_{\acute{e}t}(X,\mathcal{F}).\]
\end{prop}

\begin{proof}
    The first assertion is proved in \cite[Corollary A.1.4]{Sti02b}, so we omit the proof. Note that the cohomology of a locally constant sheaf on \(X_{f\acute et}\) is the group cohomology of the corresponding \(\pi_1(X)\)-module. For the second assertion, consider the following commutative diagram of sites:
    \[
    \begin{tikzcd}[ampersand replacement=\&]
        X_{\acute et} \arrow[r, "f"] \arrow[d] \&  
        K_{\acute et}\arrow[d, "\simeq"] \\
        X_{f\acute et} \arrow[r, "f"] \& 
        K_{f\acute et}
    \end{tikzcd}.
    \]
    This induces a morphism between spectral sequences
    \[ E_2^{p,q}:=\HH^p(K, \HH^q(\pi_1(X_{\widebar{K}}), \mathcal{F}))\to F_2^{p,q}:=\HH^p_{\acute et}(K, \HH^q_{\acute et}(X_{\widebar{K}}, \mathcal{F})). \]
    See, for example, \cite[Section 1]{Don05}. Since we have this morphism, \(F^1\HH^2(\pi_1(X), \mathcal{F})\) maps into \(F^1\HH^2_{\acute et}(X, \mathcal{F})\). Note that \(F^1\HH^2(\pi_1(X), \mathcal{F})\) (resp. \(F^1\HH^2_{\acute et}(X, \mathcal{F})\)) is an extension of \(E_\infty^{1,1}\) by \(E_\infty^{2,0}\) (resp. \(F_\infty^{1,1}\) by \(F_\infty^{2,0}\)). Since \(E_2^{p,q}\simeq F_2^{p,q}\) for \(q\le 1\), we have \(E_\infty^{1,1}\simeq F_\infty^{1,1}\) and \(E_\infty^{2,0}\simeq F_\infty^{2,0}\). Therefore, this induces an isomorphism on \(F^1\). The proof of the third assertion is almost identical, so we omit it.
\end{proof}

Now let \(X\) be a smooth proper variety over \(K\), let \(L\in\Pic(X)\), and let \(n\) be a positive integer prime to the characteristic of \(K\). If \(c_1(L_{\bar{K}})\in\HH^2(X_{\bar{K}},\mu_n)\) is zero, or equivalently, \(L_{\bar{K}}\in n\cdot\Pic(X_{\bar{K}})\), then \(c_{1,n}(L)\in F^1\HH^2_{\acute et}(X,\mu_n)\). Note that this is automatic if \(L\in\Pic^0(X)\). Hence, by \cref{prop:F1 comparison}, we may regard \(c_{1,n}(L)\) as an element of \(F^1\HH^2(\pi_1(X),\mu_n)\), or equivalently as an element of \(F^1\HH^2(\pi_1(X)^{(\mathrm{ab})},\mu_n)\).

\begin{defn}\label{defn:Chern class}
    Let \(X,K,L,n\) be as above, and let \(s:G_K\to\pi_1(X)^{(\mathrm{ab})}\) be a section. Then the \textbf{Chern class obstruction} of \(s\) with respect to \(L\) is defined to be \(s^\ast c_{1,n}(L)\in \HH^2(K, \mu_n) \) where \(c_1(L)\) is viewed as an element of \(\HH^2(\pi_1(X)^{(\mathrm{ab})},\mu_n)\).
\end{defn}

We call this an obstruction because, if \(s\) is \textbf{geometric}, i.e. arises from a \(K\)-rational point, then the Chern class obstruction vanishes for every \(L\in\Pic(X)\) and every \(n\) satisfying the above conditions.

Let \(X\) be a proper variety over a field \(K\) such that \((n,\operatorname{char}K)=1\), let \(s_O:G_K\to\pi_1(X)^{(\ab)}\) be a fixed section, and let \(s:G_K\to\pi_1(X)^{(\ab)}\) be another section. Reducing the geometric part modulo \(n\), we obtain a class
\[[s-s_O]\in\HH^1(K,\pi_1(X_{\widebar{K}})^{\ab}\otimes\Z/n).\]
Let \(L\) be a line bundle on \(X\) such that \(M^{\otimes n}\simeq L|_{X_{\widebar{K}}}\) for some \(M\in\Pic(X_{\widebar{K}})\). Then, as before, \(c_1(L)\in F^1\HH^2(X,\mu_n)=F^1\HH^2(\pi_1(X)^{(\ab)},\mu_n)\). Define \(\delta L\) to be the cocycle
\[\sigma\in G_K\longmapsto M^\sigma\otimes M^{-1}\in\HH^1(X_{\widebar{K}},\mu_n)=\Pic(X_{\widebar{K}})[n]=\operatorname{Hom}(\pi_1(X_{\widebar{K}})^{\ab}\otimes\Z/n,\mu_n)\]
appearing in \cref{prop:chern class}. The last equality follows from the standard theory of abelian covers. We now have the following general comparison between Chern class obstructions and the Weil pairing.

\begin{thm}\label{thm:Chern vs Weil}
    Under the above assumptions,
    \[(s^\ast-s_O^\ast)c_1(L)=e(\delta L\cup[s-s_O]).\]
\end{thm}

\Cref{prop:coin} follows from \cref{thm:Chern vs Weil} by applying it to an abelian variety \(A\) and \(L\in\Pic^0(A)\), since such an \(L\) automatically satisfies the assumption, \(\delta L\) coincides with the image of \(L\in A^t(K)\) under the Kummer map \(A^t(K)\to\HH^1(K,A^t[n])\), and \(s_O^\ast c_1(L)=0\) in this case.

\begin{proof}
    By \cref{prop:chern class}, \(\eta(c_1(L))=\delta L\), where \(\eta\) is the map from the Leray spectral sequence. By \cref{prop:F1 comparison} and its proof, there is a morphism from the Hochschild--Serre spectral sequence to the Leray spectral sequence inducing an isomorphism on \(F^1\). Thus, we may view \(c_1(L)\) as a class in \(\HH^2(\pi_1(X)^{(\ab)},\mu_n)\), and under this identification its image under \(\eta\) is \(\delta L\). Applying \cref{prop:group coh calc} with \(\Pi=\pi_1(X)^{(\ab)}\), \(G=G_K\) and \(W=\mu_n\), we obtain the claim.
\end{proof}

\section{Chern class obstructions and geometric sections}\label{sec:section}

Here, let \(X\) be a smooth proper variety over a field \(K\), and let \(s:G_K\to \pi_1(X)^{(\mathrm{ab})}\) be a section. The following is a generalization of the notion of being (uniformly) orthogonal to \(\Pic^0 \) in \cite{Sai11, Sai12}. Note that for any finite separable extension \(F/K\), the restriction \(s|_{G_F}\) defines a section of \(X_F\).

\begin{defn}\label{defn:good section}
We say that a section \(s:G_K\to\pi_1(X)^{(\mathrm{ab})}\) of the geometrically abelian fundamental group is \textbf{orthogonal to \(\Pic^0\)} (resp. \textbf{orthogonal to \(\Pic\)}) if every Chern class obstruction (\cref{defn:Chern class}) vanishes for all \(n\) prime to \(\operatorname{char}K\) and all \(L\in\Pic^0(X)\) (resp. all \(L\in\Pic(X)\) such that \(c_{1,n}(L_{\bar{K}})\) vanishes). We say that \(s\) is \textbf{uniformly orthogonal to \(\Pic^0\)} (resp. \textbf{uniformly orthogonal to \(\Pic\)}) if \(s|_{G_F}\) is orthogonal to \(\Pic^0\) (resp. orthogonal to \(\Pic\)) for every finite separable extension \(F/K\). 
\end{defn}

Before proceeding, we note that, over a \(p\)-adic field, orthogonality of a section is equivalent to uniform orthogonality. This was already proved in \cite[Proposition 1.6.7]{Sai12} and \cite[Lemma 3.1]{Sai21} in the case of curves, but the same argument works in general. For the reader's convenience, we include the proof.

\begin{prop}\label{prop:uniform ortho}
    Let \(K\) be a \(p\)-adic field. Then a section \(s:G_K\to\pi_1(X)^{(\mathrm{ab})}\) is uniformly orthogonal to \(\Pic^0\) (resp. \(\Pic\)) if it is orthogonal to \(\Pic^0\) (resp. \(\Pic\)).
\end{prop}

\begin{proof}
    Let \(F/K\) be a finite extension, \(f:X_F\to X\) the projection, and fix \(n\geq 1\). Let \(L\in\operatorname{Pic}^0(X_F)\) (resp. \(L\in\operatorname{Pic}(X_F)\) with \(c_1(L_{\overline K})=0\) in \(H^2(X_{\overline K},\mu_n)\)). The norm \(N_fL\) satisfies the corresponding condition on \(X\), since its pullback to \(X_{\bar{K}}\) is the tensor product of the conjugates of \(L\).

    Let \(c_1(L)\in H^2(\pi_1(X_F)^{(\mathrm{ab})},\mu_n)\) be the Chern class. By compatibility of Chern classes with norms and corestriction, \(\operatorname{Cor}_f c_1(L)=c_1(N_fL)\). Orthogonality over \(K\) and the compatibility of corestriction with pullback along sections therefore give
    \[\operatorname{Cor}_{F/K}\bigl((s_F)^\ast c_1(L)\bigr)=(s)^\ast\operatorname{Cor}_f c_1(L)=0.\]
    Note that \(\operatorname{Cor}_{F/K}:H^2(G_F,\mu_n)\to H^2(G_K,\mu_n)\) is identified with the identity on \(\frac1n\mathbb Z/\mathbb Z\), hence is injective. Thus \((s_F)^\ast c_1(L)=0\), and the assertion follows.
\end{proof}

\begin{thm}\label{thm:comp}
    Let \(K\) be a \(p\)-adic field or \(\R\), and \(A/K\) be an abelian variety. Then a section \(s:G_K\to \pi_1(A)\) is orthogonal to \(\Pic^0 \)  if and only if it is geometric.
\end{thm}

\begin{proof}[Proof of \cref{thm:comp}]
    The geometricity of \(s\) clearly implies its orthogonality to \(\Pic^0 \). We prove the converse. The local Tate duality for abelian varieties asserts that the Weil pairing
    \[ \HH^1(K, A[n])\times \HH^1(K, A^t[n])\to \HH^2(K, \mu_n) \]
    induces a perfect pairing
    \[ \HH^1(K, A)[n]\times A^t(K)/nA^t(K)\to  \HH^2(K, \mu_n)\simeq \Z/n. \]
    If \(K=\R\), then we use \(\pi_0(A(\R))\) instead and \(\HH^2(\R, \mu_n)=\Z/\text{gcd}(2,n)\). Both statements follow from the arithmetic duality theorems for local fields; see \cite{Tat63, Lic69} for the original references and \cite[Section I.3]{Mil06} for a summary.

    By \cref{prop:coin}, the vanishing of \(g_{s,n}\) for all \(n\) implies the vanishing of \(f_{s,n}\) for all \(n\). This implies that the image of \([s-s_O]\in \HH^1(K, A[n])\) in \(\HH^1(K, A)[n]\) is zero. Taking inverse limits, we see that the map
    \[ \HH^1(K, \T A)\to \T\HH^1(K, A) \]
    sends \([s-s_O]\) to zero. By the Kummer exact sequence
    \[ 0\to \widehat{A(K)}= A(K)\to \HH^1(K,\T A)\to \T\HH^1(K, A)\to 0, \]
    where for \(K=\R\) we use \(\pi_0(A(\R))\) instead of \(A(K)\), it follows that \([s-s_O]\) comes from \(A(K)\). The equality \(\widehat{A(K)}= A(K)\) (resp. \(\widehat{A(K)}= \pi_0(A(K))\)) over \(p\)-adic fields (resp. \(\R\)) holds by the group structure of \(A(K)\) (cf. \cite{Mat55}). This shows that \(s\) is geometric.
\end{proof}

Next, we state the following corollary of \cref{thm:comp}. 

\begin{cor}\label{cor:torsor}
    Let \(X\) be a torsor under an abelian variety over a \(p\)-adic field \(K\). Then a section \(s:G_K\to \pi_1(X)\) is orthogonal to \(\Pic^0 \) if and only if it is geometric.
\end{cor}

The proof of \cref{cor:torsor} is obtained from \cref{thm:comp} by a Galois descent argument. The following lemma, which is a straightforward generalization of \cite[Proposition 4]{Sti11}, takes care of this Galois descent. For any variety \(X/K\), we define a presheaf \(\mathscr{S}_{X}\) on \(\text{Spec }K_{\acute et}\) by
\[  \mathscr{S}_{X}(F):=\{\text{sections of }\pi_1(X_F)\to G_F\}/\{\text{conjugacy}\}. \]

\begin{prop}\label{prop:sheaf}
    Let \(K\) be a \(p\)-adic field or a finitely generated field, and let \(X\) be a smooth proper variety of positive dimension such that the first étale homology of any finite étale cover of \(X_{\widebar{K}}\) is torsion-free. Then \(\mathscr{S}_{X}\) is an etale sheaf.
\end{prop}

\begin{proof}
    First, we prove that for any section \(s:G_K \to \pi_1(X)\), the induced action of \(G_K\) on \(\pi_1(X_{\bar{K}})\) has no nontrivial fixed point. Suppose that \(g\) is a nontrivial fixed point. Choose an open subgroup \(U\) such that \(g \in U\) but \(g \notin [U,U]\). This can be constructed as follows: choose an open normal subgroup \(V\) that does not contain \(g\), and let \(U\) be the subgroup generated by \(V\) and \(g\). Then \(U/V\) is cyclic, so \([U,U] \le V\). 
    
    Now choose a finite extension \(F/K\) such that \(U\) is invariant under the conjugation action of \(s(G_F)\).  Then
    \[ 1 \to U \to G_F \ltimes U \to G_F \to 1 \]
    is the fundamental exact sequence corresponding to the \(F\)-model of the cover associated with \(U\). By assumption, \(G_F\) still fixes \(g\). Since \(g \notin [U,U]\), the abelianization \(U^\text{ab}\) contains a nonzero \(G_F\)-invariant element. But by the assumption and \cite[Proposition 4.3]{MT03}, \(U^\text{ab}\) is the Tate module of some abelian variety. This contradicts \cite[Lemma 3]{Sti11}. Note that \cite[Lemma 3]{Sti11} also holds over finitely generated fields, by the Mordell-Weil theorem.
    
    Now we prove that, for any finite extension \(F/K\), the normalizer of \(s(G_F)\) is contained in \(s(G_K)\). We first prove that \(N(s(G_F)) \cap \pi_1(X_{\bar{K}}) = \{1\}\). Suppose not, and let \(g\) be an element of this intersection. Consider the action of \(g\) on \(s(G_F)\). Since \(g\) maps to \(1\) under the projection map \(\pi_1(X)\to G_K\) and \(s(G_F)\to G_K\) is injective, this action is trivial. Hence \(g\) lies in the centralizer of \(s(G_F)\), which contradicts the previous paragraph. Let \(N\) be the normalizer of \(G_F\) in \(G_K\). Then \(s(N) \le N(s(G_F))\). Since \(\pi(N(s(G_F))) \le N\), this completes the proof.

    Finally, we prove the statement of the proposition. This amounts to showing that 
    \[ \mathscr{S}_{X}(K) \to \mathscr{S}_{X}(F)^{\text{Gal}(F/K)} \]
    is a bijection for any Galois extension \(F/K\). If a section \(s:G_F \to \pi_1(X)\) is \(\text{Gal}(F/K)\)-invariant, then by definition \(N(s(G_F))\) surjects onto \(G_K\). Hence \(N(s(G_F))\) determines a section in \(\mathscr{S}_{X}(K)\). This proves surjectivity. Injectivity follows from the same argument.
\end{proof}

\begin{proof}[Proof of \cref{cor:torsor}]
    Since \(X\) is a torsor over an abelian variety, \(a:X\simeq \Alb_X^1\) and \(s=s^a\). Therefore, we identify them in this proof. The implication from geometricity to orthogonality is immediate. Now suppose that \(\Alb_X^1(K)\neq\emptyset\), so that \(\Alb_X^1\) is an abelian variety, and that \(s\) is orthogonal to \(\Pic^0\). Then \(s^a\) is geometric by \cref{thm:comp}.
    
    Now consider the case where \(s\) is orthogonal to \(\Pic^0 \). By \cref{prop:uniform ortho}, \(s\) is also uniformly orthogonal. Choose a finite Galois extension \(F/K\) such that \(\Alb_X^1(F)\ne \emptyset\). Then by the previous argument, \(s^a|_{G_F}\) is a geometric section. By \cref{prop:sheaf}, \(\mathscr{S}_{\Alb_X^1}\) is an etale sheaf, and \(\Alb_X^1\hookrightarrow \mathscr{S}_{\Alb_X^1}\) is an injective morphism of sheaves. Since \(s^a|_{G_F}\) lies in the image of this morphism, \(s^a\) also lies in the image. Therefore, \(s^a\) is geometric. 
\end{proof}

We now explain how \cref{cor:torsor} can be applied to the study of sections of a general variety. Let \(X/K\) be a smooth proper variety, and let \(a:X\to \Alb_X^1\) be the Albanese map. If \(s:G_K\to \pi_1(X)^{(\ab)}\) is a geometrically abelian section, then by composing it with \(a_\ast\) we obtain a section
\[ s^{\text{a}}:G_K\to \pi_1(\Alb_X^1) \]
of \(\pi_1(\Alb_X^1)\to G_K\).  We note the following observation: The Albanese map transfers orthogonality to \(\Pic^0\) from \(X\) to its Albanese torsor, where \cref{cor:torsor} applies.

\begin{prop}\label{prop:good section relation}
    \(s\) is orthogonal to \(\Pic^0 \) (resp. uniformly orthogonal to \(\Pic^0 \))  if and only if \(s^\text{a}\) is orthogonal to \(\Pic^0 \) (resp. uniformly orthogonal to \(\Pic^0 \)).
\end{prop}

\begin{proof}
     By \cref{thm:albanese} (2), we have an isomorphism \( a_2^\ast :\Pic^0(\Alb_X^1) \simeq \Pic^0(X)\). Moreover, for any line bundle \(L\in \Pic^0(\Alb_X^1)\),
    \[ s^\ast c_1(a^\ast L)=(s^\text{a})^\ast c_1(L), \]
    so \(s\) is (uniformly) orthogonal to \(\Pic^0 \) if and only if \(s^\text{a}\) is (uniformly) orthogonal to \(\Pic^0 \).
\end{proof}

All in all, we obtain the following result from \cref{cor:torsor}, which implies \cref{thm:main 2}. 

\begin{cor}\label{cor:comp}
    Let \(K\) be a \(p\)-adic field, and let \(X\) be a smooth proper variety. Then a section \(s:G_K\to \pi_1(X)^{(\text{ab})}\) is (uniformly) orthogonal to \(\Pic^0 \) if and only if \(s^{\text{a}}\) is geometric. 
\end{cor}

\begin{proof}[Proof of \cref{thm:main 2}]
The equivalence of (1) and (2) follows from Proposition 4.2, while the equivalence with (3) follows from \cref{cor:comp}.
\end{proof}

\begin{rmk}
    One may ask what can be said in the case of number fields. In this setting, sections orthogonal to \(\Pic^0 \) are, in general, much weaker than geometric sections. For simplicity, let us focus on the case of an abelian variety \(A\) over a number field \(K\). Let \(v\) be a place of \(K\), and for any section \(s:G_K\to \pi_1(A)\), let \(s_v\) be the composition of \(s\) with an embedding \(G_{K_v}\hookrightarrow G_K\). Then
    \[  s \text{ geometric} \Rightarrow s_v\text{ geometric} \Leftrightarrow s_v\text{ orthogonal to \(\Pic^0 \)} \Rightarrow s\text{ orthogonal to \(\Pic^0 \)}. \]
    The first implication is trivial, the equivalence in the middle is \cref{thm:comp}, and the last implication follows from the Albert--Brauer--Hasse--Noether theorem. However, the converse directions of the first and third implications are far from true. If the rank of \(A\) is at least \(1\), then any element in the image of
    \[ \widehat{A(K)}\setminus A(K)\to \HH^1(K, \T A) \]
    gives a section \(s\) that is not geometric, but for which every \(s_v\) is geometric. As for the converse of the third implication, the obstruction is that \(A_{K_v}\) generally admits many more line bundles than \(A\). In particular, if \(A^t(K)=0\), then the condition of being orthogonal to \(\Pic^0 \) is vacuous.
    
    However, if we also consider \(\Pic\), and not only \(\Pic^0\), then even over general fields the vanishing of Chern class obstructions has several implications. We discuss this in \cref{sec:brauer}.
\end{rmk}

\section{Relative Brauer group}\label{sec:brauer}

In this section, we study the relation between Chern class obstructions and relative Brauer groups. The results of this section have the advantage of holding over more general fields and will be used in \cref{sec:proof}. Recall that, for a variety \(X\), the \textit{relative Brauer group} is defined by
\[\Br(X/K):=\ker\left(\Br(K)\to\Br(X)\right).\]

\begin{thm}\label{thm:relative brauer}
    Let \(X/K\) be a smooth proper variety. If there exists a section \(s:G_K\to \pi_1(X)^{(ab)}\) orthogonal to \( \Pic \), then the relative Brauer group satisfies
    \[ \Br(X/K)\text{ is }\begin{cases}
            \text{trivial} & \text{if } \operatorname{char} K=0,\\
            l\text{-primary} & \text{if } \operatorname{char} K=l.
        \end{cases} \]
    Moreover, if \(X\) is a curve, then
    \[ \text{Index}(X)/\text{Period}(X)=\begin{cases}
            1 & \text{if } \operatorname{char} K=0,\\
            \text{a power of }l & \text{if } \operatorname{char} K=l.
        \end{cases} \]
\end{thm}

\begin{proof}
    The idea is that the only classes in \(\HH^2(X,\mu_n)\) that disappear in \(\Br(X)\) are Chern classes of line bundles, while orthogonality prevents a nonzero class coming from \(\HH^2(K,\mu_n)\) from being such a Chern class. It is enough to prove that for any \(n\) prime to the characteristic of \(K\), the composite
    \[ \HH^2(K, \mu_n)\to \HH_{\acute et}^2(X, \mu_n)\to \Br(X) \]
    is injective, where the second map is induced by the inclusion \(\mu_n\hookrightarrow \mathbb{G}_m\). The long exact sequence associated to the Kummer exact sequence for \(\mathbb{G}_m\) gives a short exact sequence
    \[ 0\to \Pic(X)\otimes \Z/n\xrightarrow{c_1} \HH_{\acute et}^2(X, \mu_n)\to \Br(X)[n]\to 0. \]
    Hence, the kernel of the map \(\HH_{\acute et}^2(X, \mu_n)\to \Br(X)\) is the subgroup generated by the first Chern classes of line bundles. Therefore, it is enough to show that the image of \(c_1\) has trivial intersection with the image of \( \HH^2(K, \mu_n) \).

    Suppose not. Then there exists a line bundle \(L\) such that \(c_1(L)\) is nonzero and lies in the image of \(\HH^2(K, \mu_n)\). By \cref{prop:F1 comparison}, the map \(\HH^2(K, \mu_n)\to \HH_{\acute et}^2(X, \mu_n)\) factors through \(\HH^2(K, \mu_n)\to \HH^2(\pi_1(X)^{\text{(ab)}}, \mu_n)\), we have \(s^\ast c_1(L)=0\) because \(s\) is orthogonal to \(\Pic\). However,
    \[ \HH^2(K, \mu_n)\to \HH^2(\pi_1(X)^{\text{(ab)}}, \mu_n)\xrightarrow{s^\ast} \HH^2(K, \mu_n) \]
    is the identity map. Therefore, this implies \(c_1(L)=0\), a contradiction. The last assertion follows from \cite[Section 2]{Sti10}.
\end{proof}

Indeed, over \(p\)-adic fields, we have the following partial converse.

\begin{prop}\label{prop: Brauer converse}
    Let \(X\) be a smooth proper variety over a \(p\)-adic field \(K\). If \(\pi_1(X)^{(\ab)}\to G_K\) admits a section and \(\Br(X/K)=0\), then it admits a section that is orthogonal to \(\Pic\).
\end{prop}

For \cref{prop: Brauer converse}, we need a slight variation of the local Tate duality used above. Over a \(p\)-adic field, there is a canonical perfect \(G_K\)-equivariant pairing
\[\Pic(X_{\overline{K}})[n]\times\pi_1(X_{\overline{K}})^\ab/n\to\mu_n.\]
Therefore, local Tate duality gives a perfect pairing
\[\HH^1(K,\Pic(X_{\overline{K}})[n])\times\HH^1(K,\pi_1(X_{\overline{K}})^\ab/n)\to\Br(K)[n]\simeq\Z/n.\]
Passing to the direct limit in the first factor and the inverse limit in the second factor gives a natural pairing
\[\HH^1(K,\Pic(X_{\overline{K}})_{\mathrm{tors}})\times\HH^1(K,\pi_1(X_{\overline{K}})^\ab)\to\Br(K)\simeq\Q/\Z.\]
Indeed, if we regard \(\HH^1(K,\Pic(X_{\overline{K}})_{\mathrm{tors}})\) as a discrete torsion group and endow \(\HH^1(K,\pi_1(X_{\overline{K}})^\ab)\) with the profinite topology, then these two groups are Pontryagin dual to each other. The point is that the set of conjugacy classes of sections is naturally a torsor under \(\HH^1(K,\pi_1(X_{\overline{K}})^\ab)\).

\begin{proof}
    The proof is a correction argument. Starting from an arbitrary section \(s_0\), its Chern class obstructions define a partial character on \(H^1(K,\Pic(X_{\bar K})_{\mathrm{tors}})\). We extend this character and use local Tate duality to obtain a class \(c\) whose translation cancels all Chern class obstructions.
    
    Let \(s_0:G_K\to\pi_1(X)^{(\ab)}\) be any section. Since \(\Br(X/K)=0\), we have \(\Pic(X)=\Pic(X_{\overline{K}})^{G_K}\) by \cite[Section~2]{Sti10}. Therefore,
    \[(n\cdot\Pic(X_{\overline{K}}))^{G_K}=(n\cdot\Pic(X_{\overline{K}}))\cap\Pic(X_{\overline{K}})^{G_K}=(n\cdot\Pic(X_{\overline{K}}))\cap\Pic(X),\]
    which is precisely the subgroup of line bundles \(L\in\Pic(X)\) such that \(c_{1,n}(L_{\overline{K}})=0\); cf. \cref{defn:Chern class}. From the exact sequence
    \[0\to\Pic(X_{\overline{K}})[n]\to\Pic(X_{\overline{K}})\to n\cdot\Pic(X_{\overline{K}})\to 0,\]
    we obtain an injection
    \[(n\cdot\Pic(X_{\overline{K}}))^{G_K}/n\cdot\Pic(X)\hookrightarrow\HH^1(K,\Pic(X_{\overline{K}})[n]).\]
    Moreover, these injections are compatible with the transition maps, giving a commutative diagram
    \[
    \begin{tikzcd}
    (n\cdot \Pic(X_{\overline{K}}))^{G_K}/ n\cdot \Pic(X)
       \arrow[r, hook]
       \arrow[d, "\times m"']
    &
    \HH^1\!\left(G_K,\Pic(X_{\overline K})[n]\right)
       \arrow[d]
    \\
    (nm\cdot \Pic(X_{\overline{K}}))^{G_K}/ nm\cdot \Pic(X)
       \arrow[r, hook]
    &
    \HH^1\!\left(G_K,\Pic(X_{\overline K})[nm]\right).
    \end{tikzcd}
    \]

    Let \(F^1_{\Pic}\) be the colimit of the groups on the left-hand side. Then \(F^1_{\Pic}\hookrightarrow\HH^1(K,\Pic(X_{\overline{K}})_{\mathrm{tors}})\). Moreover, \(s_0^\ast c_{1,n}(-)\) defines a homomorphism from \((n\cdot\Pic(X_{\overline{K}}))^{G_K}/n\cdot\Pic(X)\) to \(\Br(K)[n]\subseteq\Br(K)\simeq\Q/\Z\). These homomorphisms are compatible with the transition maps and hence define a homomorphism \(F^1_{\Pic}\to\Q/\Z\). Since \(\Q/\Z\) is an injective abelian group, this homomorphism extends to \(\HH^1(K,\Pic(X_{\overline{K}})_{\mathrm{tors}})\). By local Tate duality, there exists \(c\in\HH^1(G_K,\pi_1(X_{\overline{K}})^\ab)\) inducing the same homomorphism. Let \(s\) be the section corresponding to \(-c\), i.e. \(s=s_0-c\). By \cref{thm:Chern vs Weil}, the contribution of \(c\) cancels \(s_0^\ast c_{1,n}(L)\) for every \(n\) and every \(L\) under consideration. Therefore, \(s\) is orthogonal to \(\Pic\).
\end{proof}

In particular, we obtain the following corollary, which implies \cref{thm:newmain}. Since we use \cref{lem:albanese lifting}, which is stated in \cref{sec:proof}, we postpone its proof to that section.

\begin{cor}\label{cor:period index}
    \begin{enumerate}
        \item If \(G^{(\text{ab})}_{K(X)}\to G_K\) admits a section, then
        \[ \Br(X/K)\text{ is }\begin{cases}
            \text{trivial} & \text{if } \operatorname{char} K=0,\\
            l\text{-primary} & \text{if } \operatorname{char} K=l.
        \end{cases} \]
        In particular, if \(X\) is a curve, then
        \[ \text{Index}(X)/\text{Period}(X)=\begin{cases}
            1 & \text{if } \operatorname{char} K=0,\\
            \text{a power of }l & \text{if } \operatorname{char} K=l.
        \end{cases} \]
        \item Let \(K\) be a \(p\)-adic field or a finitely generated field. Then for an open subset \(U\subseteq X\) satisfying (VP1), if there exists a section \(s_U:G_K\to \pi_1(U)^{\text{(ab)}}\), then the same conclusion as in (1) holds. Moreover, if \(K\) is a \(p\)-adic field and \(X\) is a curve, then \(\text{Index}(X)=\text{Period}(X)=1\).
    \end{enumerate}
\end{cor}

\subsection{Further consequences of an abelian section.}

We record a further consequence that does not require passing to a sufficiently small open subset. Even the existence of a section of \(\pi_1(X)^{(\mathrm{ab})}\to G_K\) already imposes restrictions on the relative Brauer group.

Recall that the Leray spectral sequence for \(X\to K\) with coefficients in \(\mathbb{G}_m\) gives
\[ 0\to \Pic(X)\to \Pic(X_{\widebar{K}})^{G_K}\xrightarrow{d_2^{0,1}}\Br(K)\to \Br(X). \]
By \cite[Appendix B]{Sti10}, \(d_2^{0,1}\) is the Brauer obstruction map. Hence, the relative Brauer group is the image of the Brauer obstruction map.

\begin{thm}\label{thm:torsion line bundle}
    Let \(X\) be a smooth proper variety over a field \(K\) such that there exists a section \(s:G_K\to \pi_1(X)^{(\text{ab})}\). Then every torsion element of \(\Pic(X_{\widebar{K}})^{G_K}\) whose order is invertible in \(K\) has trivial Brauer obstruction.
\end{thm}

Note that this generalizes \cite[Proposition 12]{Sti10}. More precisely, we weaken the assumption from the existence of a section of the full fundamental group to that of the geometrically abelian fundamental group.

\begin{proof}
    The proof uses a technique similar to that of \cref{prop:F1 comparison}. Let \(E_{2}^{p,q}\) (resp. \(F_2^{p,q}, G_2^{p,q}\)) be the Leray spectral sequence for \(X\to K\) (resp. the Hochschild--Serre spectral sequence for \( \pi_1(X)\twoheadrightarrow G_K \), \( \pi_1(X)^{(\text{ab})}\twoheadrightarrow G_K \)) with coefficients in \(\mu_n\) for some \(n\) which is invertible in \(K\). Then there exist morphisms of spectral sequences
    \begin{equation}\label{eqn:3ss}
        G_2^{p,q}\to F_2^{p,q}\to E_2^{p,q}.
    \end{equation}
    The morphism \(G_2^{p,q}\to F_2^{p,q}\) is naturally induced from the diagram
    \[
    \begin{tikzcd}[ampersand replacement=\&]
         \pi_1(X) \arrow[d]\arrow[r, two heads] \& G_K \arrow[d, "\text{id}"]\\
        \pi_1(X)^{(\text{ab})} \arrow[r, two heads] \& G_K
    \end{tikzcd}.
    \]
    and \(F_2^{p,q}\to E_2^{p,q}\) is defined as in the proof of \cref{prop:F1 comparison}. Since
    \[ \HH_{\acute et}^1(X_{\widebar{K}}, \mu_n)\simeq \HH^1(\pi_1(X_{\widebar{K}}), \mu_n)\simeq \HH^1(\pi_1(X_{\widebar{K}})^{\text{ab}}, \mu_n)\simeq \Pic(X_{\widebar{K}})[n], \]
    \cref{eqn:3ss} induces an isomorphism for \(q\le 1\). Therefore, \cref{eqn:3ss} induces the following map between portions of the five-term exact sequences:
    \[
    \begin{tikzcd}[ampersand replacement=\&]
         \HH^0(K, \Pic(X_{\widebar{K}})[n]) \arrow[d, "\simeq"]\arrow[r, "d_{2, G}^{0,1}"] \& \Br(K)[n] \arrow[d, "\simeq"]\arrow[r] \&  \HH^2(\pi_1(X)^{(\text{ab})}, \mu_n) \arrow[d] \\
         \HH^0(K, \Pic(X_{\widebar{K}})[n]) \arrow[d, "\simeq"]\arrow[r, "d_{2, F}^{0,1}"] \& \Br(K)[n] \arrow[d, "\simeq"]\arrow[r] \& \HH^2(\pi_1(X), \mu_n) \arrow[d] \\
         \HH^0(K, \Pic(X_{\widebar{K}})[n]) \arrow[r, "d_{2, E}^{0,1}"] \& \Br(K)[n] \arrow[r]\& \HH^2(X, \mu_n).
    \end{tikzcd}
    \]

    The map \(d_{2, E}^{0,1}\) is equal to \(d_2^{0,1}|_{\HH^0(K, \Pic(X_{\widebar{K}})[n])}\), by comparing the Leray spectral sequences with coefficients in \(\mu_n\) and in \( \mathbb{G}_m \). Since the upper left map \(\Br(K)[n]\to \HH^2(\pi_1(X)^{(\text{ab})}, \mu_n) \) admits a retraction given by \(s^\ast\), we have \(d_{2, G}^{0,1}=0\). Hence, by the diagram, \(d_{2, E}^{0,1}=0\). Thus \(d_2^{0,1}\), the Brauer obstruction map, vanishes on \(\HH^0(K, \Pic(X_{\widebar{K}})[n])=\Pic(X_{\widebar{K}})^{G_K}[n]\).
\end{proof}

\begin{cor}\label{cor: torsion line bundle}
    Keep all the assumptions of \cref{thm:torsion line bundle}, and moreover assume \(X\) is a curve.
    \begin{enumerate}
        \item If \(K\) is a \(p\)-adic field, then the period of \(X\) is a power of \(p\).
        \item If \(K\) is a \(p\)-adic field, \(g(X)\ge 1\), and \(s\) lifts to a section of a geometrically two-step solvable fundamental group, i.e. an extension
        \[ 1\to \pi_1(X_{\widebar{K}})^{2\text{-sol}}\to \pi_1(X)^{(2\text{-sol})}\to G_K\to 1 \]
        where \(\pi_1(X_{\widebar{K}})^{2\text{-sol}}\) is the maximal \(2\)-step solvable quotient, then the index and the order of the relative Brauer group of \(X\) are also powers of \(p\). In particular, if \(p\) is odd, then the period and index of \(X\) coincide.
        \item If \(K=\Q\), \(g(X)\ge 1\), and \(s\) lifts to a section of a geometrically two-step solvable fundamental group, then the period and index of \(X\) coincide.
    \end{enumerate}
\end{cor}

\begin{proof}
    The proof is almost identical to that of the corresponding theorems in \cite[Sections 5--6]{Sti10}, with the necessary adjustments. In particular, the proof of (1) is exactly the same as that of \cite[Corollary 14]{Sti10}, using \cref{thm:torsion line bundle} in place of \cite[Proposition 12]{Sti10}.

    (2) We follow the proof of \cite[Theorem 15]{Sti10}. We need only construct a geometrically abelian cover \(X'\to X\) of even degree such that the corresponding subgroup \(N\subseteq \pi_1(X)^{(\text{ab})}\) contains the image of \(s\). Such a cover always exists since \(g(X)\ge 1\). The rest of the proof is then the same as that of \cite[Theorem 15]{Sti10}, using \cref{thm:torsion line bundle} on \(X'\).

    (3) It is enough to prove that the relative Brauer group of \(X\) vanishes. Let \( \alpha \in \Br(X/\Q) \). By (2), \( \alpha_{\Q_p} \) is \(p\)-torsion for every prime \(p\). Moreover, by \cite{Wic14}, \(X\) admits an \(\R\)-point, so \(\alpha_\R=0\). Therefore, \(\alpha=0\), and hence \(\Br(X/\Q)=0\).
\end{proof}

\section{Applications}\label{sec:proof}

In this section, we give several applications of \cref{sec:section} and, in particular, prove some of the main theorems of the paper. The mechanism is simple: the conditions \textup{(VP0)} and \textup{(VP1)} force the relevant Chern class obstructions to vanish (cf. \cref{lem:albanese lifting,lem:Picard_vaishing}). Analogous vanishing phenomena for curves have appeared in \cite{Moc05,Sai11,Sai12,Sai21}. Here we establish a general statement in arbitrary dimension.

\begin{lem}\label{lem:albanese lifting}
    Let \(X\) be a smooth proper variety over a field \(K\), \(U\subseteq X\) an open subset, \(D:=X\setminus U\), and \( \langle D\rangle\subseteq \Pic(X) \) the subgroup generated by divisors supported on \(D\).

    \begin{enumerate}
        \item If \(s:G_K\to \pi_1(X)^{(\text{ab})}\) lifts to a section \(s_U:G_K\to \pi_1(U)^{(\text{ab})}\), then for any line bundle \(L\in \langle D\rangle\) and \(n\in \N\) prime to char \(K\) such that \(c_{1,n}(L_{\widebar{K}})=0\), we have \(s^{\ast} c_{1,n}(L)=0\).
        \item Let \( U\subseteq V \) be open subsets of \(X\), \(s_V:G_K\to \pi_1(V)^{(\text{ab})}\), and \(s_V^{\text{a}}:G_K\to \pi_1(\Alb_V^1)\) be the corresponding section. Assume that \(\text{char }K=0\) and the divisors supported on \(X_{\widebar{K}}\setminus V_{\widebar{K}}\) generate the N\'eron-Severi group of \(X_{\widebar{K}}\). Then the following converse of (1) holds: if \((s_V|_{G_F})^\ast c_1(L)=0\) for every finite extension \(F/K\) and \(L\in \Pic^0(X_F)\cap  \langle D_F\rangle\), then \(s_V\) lifts to a section \(s_U:G_K\to \pi_1(U)^{(\text{ab})}\).
    \end{enumerate}
\end{lem}

\begin{proof}
    (1) It is enough to prove the following: let \(s:G_K\to \pi_1(X)^{(ab)}\) be a section and \(L\) a line bundle such that \(c_1(L_{\bar{K}})\in\HH^2(X_{\bar{K}},\mu_n)\) vanishes. If \(s\) lifts to \(s_U:G_K\to \pi_1(U)^{\text{(ab)}}\) such that \(L|_U\) is trivial, then \(s^\ast c_1(L)=0\). The proof is straightforward:
    \[ s^\ast c_1(L)=s_U^\ast c_1(L|_U)=s_U^\ast c_1(\mathcal{O}_U)=0. \]

    (2) We translate the lifting problem into the splitting of the torus extension \(\Alb_U^1\to\Alb_V^1\). We then work at finite level \(T[n]\), where the extension is represented by a Chern class. Since \(T\) is quasi-split, Shapiro's lemma decomposes this obstruction into Chern class obstructions of line bundles; \cref{lem:albanese line bundle 2} allows us to compare these with the classes appearing in the hypothesis.
    
    We use notations introduced in \cref{lem:albanese line bundle 2}. Let \(T_U\), \(T_V\), and \(T\) be the kernels of \(\Alb_U^0\to \Alb_X^0\), \(\Alb_V^0\to \Alb_X^0\), and \(\Alb_U^0\to \Alb_V^0\), respectively. Then we have an exact sequence
    \[ 1\to T\to T_U\to T_V\to 1 \]
    and \(T=\prod_{j=1}^{m}\text{Res}_{F_j/K}\mathbb{G}_m\). Moreover, \(\Alb_U^1\to \Alb_V^1\) is a \(T\)-torsor.

    We want to lift a section \(s_V\) to \(s_U\). By \cref{thm:albanese} (4), this is equivalent to lifting \(s_V^{\text{a}}\) to \(s_U^{\text{a}}\). Hence, it is enough to prove that
    \[ 1\to  \pi_1(T_{\overline{K}}) \simeq  \mathbb{T}(T) \to \pi_1(\Alb_U^1)\to \pi_1(\Alb_V^1)\to 1  \]
    splits after pullback along \(s_V^{\text{a}}\).

    By passing to the limit, it is enough to prove this for each finite quotient of \(\pi_1(T)\). Since \(T[n]\) is a \(G_K\)-module, it may be regarded as a locally constant sheaf on \(\Alb_V^1\). By \cref{prop:F1 comparison}, we have
    \[ \HH^2(\pi_1(\Alb_V^1), T[n])\hookrightarrow \HH^2(\Alb_V^1, T[n]). \]
    Moreover, by \cite[Proposition 16]{BES15}, the class of the group extension
    \begin{equation}\label{eqn:ext}
        1\to  \pi_1(T)\otimes \Z/n \simeq  T[n]\to \pi_1(\Alb_U^1)/n\pi_1(T)\to \pi_1(\Alb_V^1)\to 1
    \end{equation}
    in \(\HH^2(\pi_1(\Alb_V^1), T[n])\) maps to \(c_1[\Alb_U^1]\), which is the image of \([\Alb_U^1]\in \HH^1(\Alb_V^1, T)\) under the boundary map corresponding to the Kummer exact sequence of \(T\). By abuse of notation, we also denote the extension class of \cref{eqn:ext} by \(c_1[\Alb_U^1]\).

    Now it is enough to show \((s_V^{\text{a}})^\ast c_1[\Alb_U^1]\in \HH^2(K, T[n]) \) is zero. Since \(T=\prod_{j=1}^{m}\text{Res}_{F_j/K}\mathbb{G}_m\), Shapiro's lemma implies that a \(T\)-torsor on \(\Alb_V^1\) corresponds to a collection of \(\mathbb{G}_m\)-torsors, i.e. line bundles \(L^1_j\) on \( \Alb_{V, F_j}^1 \). Moreover, \(T[n]=\prod_{j=1}^{m}\text{Res}_{F_j/K}\mu_n\), so
    \[\HH^2(K,T[n])=\prod_{j=1}^m \HH^2(F_j, \mu_n),\]
    and each component of \((s_V^{\text{a}})^\ast c_1[\Alb_U^1]\) corresponds to \((s_V^{\text{a}})^\ast c_1(L^1_j)\). Hence we need to show that they vanish. Although \(L_j^1\) is defined over \(F_j\), for simplicity we write \(s_V^{\mathrm a}\) in place of \(s_V^{\mathrm a}|_{G_{F_j}}\).

    Indeed, by the discussion before \cref{lem:albanese line bundle 2}, the line bundles \(L_j^1\) correspond to \(\Alb_{U, F_j, \chi_j}^1\). Moreover, by \cref{lem:albanese line bundle 2}, there exists \(d\in \Z_{>0}\) such that \((\Alb_{U, F_j, \chi_j}^1)^{\otimes d}\) is the pullback of an algebraically trivial line bundle \(L_j'\) on \(\Alb_{X, F_j}^1\). Let \(L_j:=a^\ast L_j'\), so that \(L_j^{1, \otimes d}=\pi^\ast L_j'\), where \(\pi:\Alb_V^1\to \Alb_X^1\). Since there exists a map \(\Alb_{U, F_j}^1\to \Alb_{U, F_j, \chi_j}^1\), the line bundle \(L_j'\) becomes trivial on \(\Alb_{U, F_j}^1\), and hence \(L_j\) is trivial on \(U_{F_j}\). Therefore, \(L_j\) corresponds to a class in \(\Pic^0(X_{F_j})\cap  \langle D_{F_j}\rangle\). By assumption, 
    \[ (s^{\text{a}}|_{G_{F_j}})^\ast c_1(L_j')=s|_{G_{F_j}}^\ast c_1(L_j)=0 \]
    where \(s\) and \(s^a\) are induced geometrically abelian sections on \(X\) and \(\Alb_X^1\) of \(s_V\). Thus \(d\cdot (s_V^{\text{a}})^\ast c_1(L_j^1)=0\).

    Note that this holds for every \(n\). Passing to the limit, we obtain the class \((s_V^{\text{a}})^\ast c_1(L_j^1)\) as an element of \( \varprojlim \HH^2(F_j, \mu_n)=\HH^2(F_j, \widehat{\Z}(1)) \). Since \(\HH^2(F_j, \widehat{\Z}(1))\) is the Tate module of \(\Br{F_j}\), it is torsion-free (see, e.g. \cite[Section 2.2]{SZ08}). Hence, the equality \(d\cdot (s_V^{\text{a}})^\ast c_1(L_j^1)=0\) implies \((s_V^{\text{a}})^\ast c_1(L_j^1)=0\). Therefore, \((s_V^{\text{a}})^\ast c_1(\Alb_U^1)=0\), and the statement follows.
\end{proof}

The second ingredient is that \((\mathrm{VP0})\) and \((\mathrm{VP1})\) make the boundary classes appearing in Lemma 6.1 dense in the relevant Picard groups.

\begin{lem}\label{lem:Picard_vaishing}
Let \(X, U, D, \langle D\rangle\) as in \cref{lem:albanese lifting}.
\begin{enumerate}
    \item If \(\widehat{\Pic^0(U)}=0\), then \(\Pic^0(X)\cap \langle D\rangle\) is dense in \(\Pic^0(X)\) for the profinite topology. Equivalently, the classes of divisors supported on \(D\) topologically generate \(\Pic^0(X)\).
    \item If \(\widehat{\Pic(U)}=0\), then \(\langle D\rangle\) is dense in \(\Pic(X)\) for the profinite topology. Equivalently, the classes of divisors supported on \(D\) topologically generate \(\Pic(X)\).
\end{enumerate}
\end{lem}

\begin{proof}
Since \(X\) is smooth, it is regular, and 
\[\Pic(X)/\langle D\rangle\simeq \Pic(U).\]
Similarly, restricting to algebraically trivial line bundles gives
\[\Pic^0(X)/\bigl(\Pic^0(X)\cap\langle D\rangle\bigr)\simeq \Pic^0(U).\]

Recall that for an abelian group \(A\) and a subgroup \(B\subseteq A\), the subgroup \(B\) is dense in \(A\) for the profinite topology if and only if \(A/B\) has trivial profinite completion. Applying this observation to \(A=\Pic^0(X), B=\Pic^0(X)\cap\langle D\rangle,\) the condition \(\widehat{\Pic^0(U)}=0\) proves (1). Applying it to \(A=\Pic(X), B=\langle D\rangle,\) the condition \(\widehat{\Pic(U)}=0\) proves (2).
\end{proof}

Combining \cref{lem:albanese lifting,lem:Picard_vaishing} with \cref{thm:main 2}, we obtain our first main application, which generalizes \cref{thm:main}.

\begin{thm}\label{thm:genmain}
    Let \(X\) be a smooth proper variety over a \(p\)-adic field \(K\), let \(s:G_K\to \pi_1(X)^{(\text{ab})}\) be a section of \(\pi_1(X)^{(\text{ab})}\to G_K\), and let \(s^a\) be the corresponding section. Choose a finite Galois extension \(F/K\) such that \(X(F)\ne \emptyset\), and an open subvariety \(U\subseteq X\) such that \(U_F\) satisfies (VP0). Then \(s^a\) is geometric whenever \(s\) lifts to a section \(s_U:G_K\to \pi_1(U)^{(\text{ab})}\).
\end{thm}

\begin{proof}
    First, consider the case where \(X(K)\ne \emptyset\). Let \(U\) be as in the statement with \(K=F\). If \(s\) lifts to \(s_U\), then by \cref{lem:albanese lifting}, we have \(s^{ \ast} c_1(L)=0\) for every \(L\in  \Pic^0(X)\cap \langle D\rangle\) where \(D:=X\setminus U\). Note that \(s^{\ast} c_1(-)\) is a homomorphism
    \[ s^{\ast} c_1(-):\Pic^0(X)\to \HH^2(K, \mu_n). \]
    By \cref{lem:Picard_vaishing}, it follows that \(s^{\ast} c_1(L)=0\) for every \(L\in \Pic^0(X)\) and every \(n\). Therefore, \(s\) is orthogonal to \(\Pic^0\). Since \(X(K)\ne \emptyset\), \cref{cor:comp} implies that \(s^\text{a}\) is geometric.

    Now consider the general case. Choose a finite Galois extension \(F/K\) such that \(X(F)\ne\emptyset\), and let \(U\) be as in the statement. Assume that \(s\) is a section that lifts to \(s_U:G_K\to \pi_1(U)^{(\text{ab})}\). Then \(s|_{G_F}\) lifts to \(s_U|_{G_F}:G_F\to \pi_1(U_F)^{(\text{ab})}\). By the previous paragraph, \(s^{\text{a}}|_{G_F}\) is geometric. By the same Galois descent argument as in the proof of \cref{cor:torsor}, it follows that \(s^\text{a}\) is geometric.
\end{proof}

\begin{proof}[Proof of \cref{thm:main}]
Suppose first that \(s\) lifts to a section
\[ s_U:G_K\to \pi_1(U)^{(\text{ab})}.\]
Then \cref{thm:genmain} implies that the original section \(s\) is geometric. Conversely, if \(s\) is geometric, then \(s\) admits a lifting to a birational section, hence in particular to a section of \(\pi_1(U)^{(\text{ab})}\). This completes the proof.
\end{proof}

\cref{lem:albanese lifting} also enables us to prove \cref{cor:period index}.

\begin{proof}[Proof of \cref{cor:period index}]
    (1) Let \(s_{\text{bir}}:G_K\to G_{K(X)}^{(\ab)}\) be a birational section, and let \(s:G_K\to \pi_1(X)^{(\text{ab})}\) be the induced section of \(\pi_1(X)^{(\text{ab})}\to G_K\). Then \(s\) is orthogonal to \(\Pic\), by \cref{lem:albanese lifting} (1). The statement now follows from \cref{thm:relative brauer}.

    (2) By \cref{lem:Picard_vaishing}, the divisors supported on \(X\setminus U\) are dense in \(\Pic(X)\). For the first assertion, it is enough to show that the induced section \(s:G_K\to \pi_1(X)^{(\text{ab})}\) is orthogonal to \(\Pic\). If \(L\in \Pic^0(X)\), then \(s^\ast c_1(L)=0\) follows from \cref{lem:albanese lifting} (1), as in the proof of \cref{thm:genmain}. For a general \(L\in \Pic(X)\), choose \(E_1,\cdots, E_m\) supported on \(X\setminus U\) such that \(L':=L-\sum_{i=1}^{m}a_iE_i\in\Pic^0(X) \) for some \(a_1,\cdots, a_m\in \Z\). If \(c_1(L_{\bar{K}})\in\HH^2(X_{\bar{K}},\mu_n)\) vanishes, then for \(i:U\hookrightarrow X\),
    \[ s^\ast c_1(L)=s_U^\ast c_1(L|_U)=s_U^\ast c_1(L'|_U)=s^\ast c_1(L')=0 \]
    since \(L'\in \Pic^0(X)\). Therefore, \(s\) is orthogonal to \(\Pic\). Hence, \cref{thm:relative brauer} completes the proof. The last assertion now follows from \cite{Lic69}, since \(X\) has period \(1\) by \cref{cor:comp}.
\end{proof}

Finally, we present an application of \cref{thm:genmain} to the computation of the generic period of a variety. Recall, as mentioned in \cref{sec:intro}, that for a smooth proper variety \(X\), the \textit{generic period} is the supremum of the periods of open subsets of \(X\).

\begin{thm}\label{thm:EW}(\cite{Wit08, EW10})
    Let \(X\) be a smooth proper variety over a \(p\)-adic field \(K\). Then the generic period of \(X\) is \(1\) if and only if
    \begin{equation}\label{eqn:elementary}
        1\to G_{\widebar{K}(X)}^{\text{ab}}\to \pi_1(K(X))^{(\text{ab})}\to G_K\to 1
    \end{equation}
    admits a section.
\end{thm}

Indeed, \cite{Wit08} proved that the generic period of \(X\) is \(1\) if and only if the \textit{elementary obstruction} of \(X\) vanishes, and \cite{EW10} proved that the elementary obstruction vanishes if and only if \cref{eqn:elementary} admits a section. Here, we prove a strengthening of \cref{thm:EW} whose proof is independent of \cite{Wit08, EW10}.

\begin{thm}\label{thm:generic period}
    Let \(X\) be a smooth proper variety over a \(p\)-adic field \(K\). Choose a finite Galois extension \(F/K\) with \(X(F)\neq\emptyset\) and \(\text{NS}(X_F)=\text{NS}(X_{\widebar{K}})\), and an open subset \(U\subseteq X\) such that \(U_F\) satisfies \textup{(VP1)}. Then the morphism \(\pi_1(U)^{(\mathrm{ab})}\to G_K\) admits a section if and only if the generic period of \(X\) is \(1\).
\end{thm}

The following lemma is essentially contained in \cite[Proof of Theorem A]{Sai21}, but we include it for convenience.

\begin{lem}\label{lem:torus}
    Let \(K\) be a \(p\)-adic field, let \(T\) be a torus over \(K\), and let \(P\) be a \(T\)-torsor over \(K\). If \( \pi_1(P)\to G_K \) admits a section, then \(P\) is a trivial \(T\)-torsor.
\end{lem}

\begin{proof}
    Let \( [P]\in \HH^1(K, T) \) be the cohomology class corresponding to \(P\). Then the boundary map in the Kummer exact sequence
    \[\delta: \HH^1(K, T)\to \HH^2(K, \T(T)) \]
    sends \([P]\) to the extension \( 1\to \pi_1(P_{\widebar{K}})\simeq \pi_1(T_{\widebar{K}}) \to \pi_1(P)\to G_K\to 1 \), by \cite[Section 2]{HS09}. Hence, if \( \pi_1(P)\to G_K \) admits a section, then \(\delta([P])=0\), so by the Kummer exact sequence, \([P]\) is divisible. However, by \cite[II.5.8, Theorem 6]{Ser01}, \(\HH^1(K, T)\) is finite, so \([P]=0\).
\end{proof}

\begin{proof}[Proof of \cref{thm:generic period}]
    By \cref{thm:albanese} (4) and \cref{lem:Picard_vaishing}, \(\pi_1(U)^{(\text{ab})}\simeq \pi_1(\Alb_U^1)\). Therefore, if the generic period of \(X\) is \(1\), then \(\Alb_U^1\) has a \(K\)-rational point. This point induces a section \(G_K\to \pi_1(\Alb_U^1)\). Via the isomorphism \(\pi_1(U)^{(\text{ab})}\simeq \pi_1(\Alb_U^1)\), this yields a section of \( \pi_1(U)^{(\text{ab})}\to G_K \).

    We now prove that if \(\pi_1(U)^{(\text{ab})}\to G_K\) admits a section \(s_U\), then the generic period of \(X\) is \(1\). Let \(s_U^{\text{a}}:G_K\to \pi_1(\Alb_U^1)\) be the associated section, and let \(s:G_K\to \pi_1(X)^{(\text{ab})}\) and \(s^{\text{a}}:G_K\to \pi_1(\Alb_X^1)\) be the sections obtained by composing with the natural pushforward maps
    \[ \pi_1(U)^{(\text{ab})}\to \pi_1(X)^{(\text{ab})}, \qquad \pi_1(\Alb_U^1)\to \pi_1(\Alb_X^1). \]
    By the same argument as in the proof of \cref{thm:genmain}, \(s\) is orthogonal to \(\Pic^0\), and hence uniformly orthogonal to \(\Pic^0\) by \cref{prop:uniform ortho}. By \cref{lem:Picard_vaishing}, the choice of \(U\), and \cref{lem:albanese lifting} (2), the section lifts to a section \(s_V:G_K\to \pi_1(V)^{(\text{ab})}\) for every smaller open subset \(V\subseteq U\). By compactness, these lifts may be chosen compatibly. Passing to the inverse limit over all such \(V\), we obtain a section \(s_{K(X)}:G_K\to G_{K(X)}^{(\text{ab})}\).
    
    At this point, the implication follows from \cref{thm:EW}. To keep our proof independent of \cite{Wit08, EW10}, we now give a direct argument. By \cref{cor:comp}, \(s^\text{a}\) is geometric. Let \(x\in \Alb_X^1(K)\) be the point corresponding to \(s^\text{a}\). Let \(V\subseteq X\) be any open subset. Then there exists a torus \(T_V\) such that \(f:\Alb_V^1\to \Alb_X^1 \) is a \(T_V\)-torsor. Let \(P\) be the fiber of \(f\) over \(x\), which is a \(T_V\)-torsor over \(K\). Then \(f\) induces an exact sequence
    \[ 1\to \pi_1(T_{V,\widebar{K}})\to \pi_1(\Alb^1_V)\to \pi_1(\Alb_X^1)\to 1. \]
    Pulling back this exact sequence along \(s^{\text{a}}\), we obtain the fundamental exact sequence
    \begin{equation}\label{eqn:fundamental}
        1\to \pi_1(T_{V,\widebar{K}})\to \pi_1(P)\to G_K\to 1
    \end{equation}
    of \(P\). Since \(s^{\text{a}}\) lifts to \(s_{K(X)}\), it in particular lifts to \(\pi_1(\Alb^1_V)\). Therefore, \cref{eqn:fundamental} splits, and hence \(P\) is trivial by \cref{lem:torus}. In particular, \(P\) admits a \(K\)-rational point, so \(\Alb_V^1\) is a trivial torsor. Hence the generic period of \(X\) is \(1\).
\end{proof}

\section{The Albanese cokernel}\label{sec:albanese}

So far, we have studied when a section of \(\pi_1(X)^{(\ab)}\to G_K\) determines a \(K\)-rational point on \(\Alb_X^1\), using the Chern class obstructions arising from \(\Pic^0\) (cf. \cref{thm:main 2}). We now show that the remaining Chern class obstructions, coming from the N\'eron-Severi group, are related to whether the induced point on \(\Alb_X^1\) is represented by a \(0\)-cycle of degree \(1\). This relation is conditional on Kai's proposed higher-dimensional analogue of Lichtenbaum duality \cite{Kai16}.

We begin with the easy direction: a \(0\)-cycle of degree \(1\) canonically produces a section whose Chern class obstructions vanish on all of \(\Pic\).

\begin{prop}\label{prop: 0 cycle orthogonal}
    Let \(z\) be a \(0\)-cycle of degree \(1\) on a smooth proper variety \(X\). Then there exists a section \(s_z:G_K\to \pi_1(X)^{(\ab)}\), orthogonal to \(\Pic\), such that the corresponding section \(s_z^a:G_K\to \pi_1(\Alb_X^1)\) is geometric and given by \(z\in \Alb_X^1(K)\).
\end{prop}

\begin{proof}
    First, we prove the existence of a section of \(\pi_1(X)^{(\ab)}\to G_K\). Write \(z=\sum_i m_i[p_i]\), where \(m_i\in \Z\), and let \(K_i:=K(p_i)\). The exact sequence
    \[1\to \pi_1(X_{\overline{K}})^{\ab}\to \pi_1(X)^{(\ab)}\to G_K\to 1\]
    defines a class \(c\in \HH^2(K,\pi_1(X_{\overline{K}})^{\ab})\). Since \(p_i\) is a \(K_i\)-rational point, \(c|_{G_{K_i}}\) vanishes, and in particular \(c\) is annihilated by \([K_i:K]\). Since \(z=\sum_i m_i[p_i]\) has degree \(1\), we have \(\sum_i m_i[K_i:K]=1\). Therefore, \(c=0\), and hence there exists a section. Let \(s_O\) be such a section.

    Now we define \(s_z:G_K\to \pi_1(X)^{(\ab)}\). For each \(p_i\), let \(s_{p_i}:G_{K_i}\to \pi_1(X_{K_i})^{(\ab)}\) be the corresponding section and consider the class
    \[[s_{p_i}-s_O|_{G_{K_i'}}]\in \HH^1(K_i',\pi_1(X_{\overline{K}})^{\ab}).\]
    where \(K_i'/K\) is the maximal separable subextension of \(K_i/K\). We may use \(K_i'\) instead of \(K_i\) since \(G_{K_i}\simeq G_{K_i'}\) and \(\pi_1(X_{K_i})^{(\ab)}\simeq \pi_1(X_{K_i'})^{(\ab)}\). Define
    \[s_z:=s_O+\sum_i m_i\cdot I_i\cdot \operatorname{cor}_{K_i'/K}[s_{p_i}-s_O|_{G_{K_i'}}],\]
    where \(I_i\) is the inseparable degree of \(K_i/K\). This is independent of the choice of \(s_O\), since
    \[1-\sum_i m_i\cdot I_i\cdot [K_i:K]_s=1-\sum_i m_i\cdot [K_i:K]=0.\]
    Moreover, for any abelian variety, the natural map \(A(K)\to \HH^1(K,\pi_1(A_{\overline{K}}))\) is a group homomorphism, and the trace map on the left-hand side commutes with corestriction on the right-hand side. Therefore, \(s_z^a\) is associated with the point \(z\in \Alb_X^1(K)\).

    Finally, we prove that \(s_z\) is orthogonal to \(\Pic\). Fix \(n\) prime to char \(K\) and \(L\) such that \(c_{1,n}(L_{\overline{K}})=0\). We use the pairing
    \[e(-\cup -):\HH^1(K,\Pic(X_{\bar{K}})[n])\times \HH^1(K,\pi_1(X_{\bar{K}})^{\ab}/n)\to \HH^2(K,\mu_n).\]
    By the standard compatibility of cup products with restriction and corestriction and \cref{thm:Chern vs Weil},
    \begin{align*}
        e(\delta L\cup \operatorname{cor}_{K_i'/K}[s_{p_i}-s_O|_{G_{K_i'}}])&=\operatorname{cor}_{K_i'/K}\left(e(\operatorname{res}_{K_i'/K}(\delta L)\cup [s_{p_i}-s_O|_{G_{K_i'}}])\right)\\
        &=-\operatorname{cor}_{K_i'/K}\left(s_O|_{G_{K_i'}}^\ast c_1(L)\right)=-[K_i:K]_s s_O^\ast c_1(L).
    \end{align*}
    Hence
    \begin{align*}
       s_z^\ast c_1(L)-s_O^\ast c_1(L)&=e(\delta L\cup [s_z-s_O])=\sum_i m_i\cdot I_i\cdot e(\delta L\cup \operatorname{cor}_{K_i'/K}[s_{p_i}-s_O|_{G_{K_i'}}])\\
       &=-\sum_i m_i\cdot I_i\cdot[K_i':K]_s s_O^\ast c_1(L)=-s_O^\ast c_1(L).
    \end{align*}
    Thus \(s_z^\ast c_1(L)=0\). Therefore, \(s_z\) is orthogonal to \(\Pic\).
\end{proof}

The converse is essentially a statement about the Albanese cokernel, i.e. the map \( \operatorname{CH}_0(X)^0\to \Alb_X^0(K). \) If \(X\) is a smooth proper curve and \(K\) is a \(p\)-adic field, then this map is \(\Pic^0(X)\to \Pic_X^0(K).\) By Lichtenbaum duality and Tate duality \cite{Tat63,Lic69}, this map is dual to
\[\HH^1(K,\Pic_X^0)\to \HH^1(K,\Pic_X),\]
whose kernel is the cokernel of the map
\[\Pic(X_{\overline{K}})^{G_K}\to \Z=\operatorname{NS}(X_{\overline{K}})^{G_K}.\]
Altogether, the cokernel of \(\Pic^0(X)\to \Pic_X^0(K)\) is Pontryagin dual to \(\operatorname{coker}(\Pic(X_{\overline{K}})^{G_K}\to \operatorname{NS}(X_{\overline{K}})^{G_K})\). Kai proposed the following statement, analogous to the preceding duality, as a higher-dimensional generalization of Lichtenbaum duality.
\[\textup{(K)}\ \operatorname{coker}\left(\operatorname{CH}_0(X)^0\to\Alb_X^0(K)\right)\text{ is Pontryagin dual to }\operatorname{coker}\left(\Pic(X_{\overline{K}})^{G_K}\to\operatorname{NS}(X_{\overline{K}})^{G_K}\right).\]
Here is an equivalent formulation of \textup{(K)}. Consider the local Tate duality pairing \[\Alb_X^0(K)\times \HH^1(K,\Pic_X^0)\longrightarrow \Br(K).\] The exact sequence relating \(\Pic\) and \(\operatorname{NS}\) gives a connecting map \(\operatorname{NS}(X_{\overline K})^{G_K}\longrightarrow \HH^1(K,\Pic_X^0),\) whose image is naturally identified with \(\operatorname{coker}\left(\Pic(X_{\overline K})^{G_K}\longrightarrow \operatorname{NS}(X_{\overline K})^{G_K}\right).\) Thus, \textup{(K)} is equivalent to the assertion that a point of \(\Alb_X^0(K)\) is represented by a \(0\)-cycle if and only if it pairs trivially with the image of \(\operatorname{NS}(X_{\overline K})^{G_K}\) in \(\HH^1(K,\Pic_X^0)\). By \cite[Proposition 3.4 and Equation (6)]{Kai16}, every point represented by a \(0\)-cycle is orthogonal to this image. The substantive content of \textup{(K)} is therefore the converse: a point of \(\Alb_X^0(K)\) that is orthogonal to the image of \(\operatorname{NS}(X_{\overline K})^{G_K}\) must arise from a \(0\)-cycle. 

As we will see, this formulation is exactly the one compatible with the Chern class obstruction: \cref{thm:Chern vs Weil} will identify the above Néron-Severi pairing with the Chern class obstruction of a suitable line bundle. This gives the following main theorem of this section, which is \cref{thm:albanese coker}.

\begin{thm}\label{thm:Kai and Sec}
    Assume that \(K\) is a \(p\)-adic field and \(X\) is a smooth projective variety of index \(1\). Then \textup{(K)} is equivalent to the following statement: the set of points of \(\Alb_X^1(K)\) represented by \(0\)-cycles of degree \(1\) is exactly the set of points of \(\Alb_X^1(K)\) associated with sections of \(\pi_1(X)^{(\ab)}\to G_K\) orthogonal to \(\Pic\).
\end{thm}

\begin{proof}
    The argument has two main ingredients. First, we identify the Néron-Severi obstruction appearing in (K) with the Chern class obstruction of a section; this is the content of Steps 1 and 2. The forward implication then follows immediately from orthogonality to \(\Pic\) (Step 3). Conversely, starting from an Albanese point annihilating the Néron--Severi obstruction, we first lift its Albanese section to \(\pi_1(X)^{(\mathrm{ab})}\) and then modify the lift so that it becomes orthogonal to \(\Pic\) (Steps 4 and 5).

    \textbf{Step 1. A diagram. }The proof uses the following diagram:

    \[
    \begin{tikzcd}[row sep=2.2em, column sep=0.55em]
    \HH^1(K,\Pic_X[n]) & \times & \HH^1(K,\pi_1(X_{\overline K})^{\mathrm{ab}}/n) \arrow[d] \arrow[drr] & & \\
    \HH^1(K,\Pic^0_X[n]) \arrow[u] \arrow[d,equal] & \times & \HH^1(K,\pi_1(\Alb^1_{X,\overline K})/n) \arrow[d,phantom,"\simeq" {description,rotate=90}] \arrow[rr] & \phantom{\qquad} & |[yshift=-1.5em,overlay]| \Br(K) \\
    \HH^1(K,\Pic^0_X[n]) \arrow[d] & \times & \HH^1(K,\pi_1(\Alb^0_{X,\overline K})/n) \arrow[urr] & & \\
    \HH^1(K,\Pic_X^0) & \times & \Alb_X^0(K) \arrow[u] \arrow[uurr] & &
    \end{tikzcd}
    \]
    
    Since \(X\) has index \(1\), we obtain an isomorphism \(\Alb_X^0\simeq \Alb_X^1\) after choosing a \(K\)-point \(z\in \Alb_X^1(K)\), which induces \(\HH^1(K,\pi_1(\Alb^0_{X,\overline K})/n)\simeq \HH^1(K,\pi_1(\Alb^1_{X,\overline K})/n)\). This identification does not depend on the choice of the \(K\)-point, since we are considering the abelianized fundamental group. The commutativity of the upper part follows from the functoriality of local Tate duality applied to \(X\to \Alb_X^1\), while the commutativity of the lower part is a standard consequence of local Tate duality.

    \textbf{Step 2. Setup and a formula. } The purpose of this step is to identify the Néron–Severi obstruction of \(y-z\) with the Chern class obstruction of \(s\); concretely, we will prove \cref{eq: the formula}. This formula is the core of the proof.
    
Now assume that \(z\) is represented by a \(0\)-cycle of degree \(1\), and let \(s\) be a section whose associated section \(s^a\) is geometric. Let \(y\in\Alb_X^1(K)\) be the point corresponding to \(s^a\). Then, under the isomorphism \(\Alb_X^0\simeq \Alb_X^1\) induced by \(z\), it also defines a section, which we denote by \(s^a_{-z}\), of \(\pi_1(\Alb_X^0)\to G_K\). Let \(s_O\) be the section of \(\pi_1(\Alb_X^0)\to G_K\) corresponding to the identity. Then, under the isomorphism \(\HH^1(K,\pi_1(\Alb^0_{X,\overline K})/n)\simeq \HH^1(K,\pi_1(\Alb^1_{X,\overline K})/n)\), the class \([s^a-s_z^a]\) maps to \([s^a_{-z}-s_O]\). Moreover, \(y-z\) maps to \([s^a_{-z}-s_O]\) under the map \(\Alb_X^0(K)\to \HH^1(K,\pi_1(\Alb^0_{X,\overline K})/n)\).

    Let \(L\in \operatorname{NS}(X_{\overline{K}})^{G_K}\). Since \(\delta_1 L\in \HH^1(K,\Pic_X^0)\) for the connecting map \(\delta_1:\operatorname{NS}(X_{\overline{K}})^{G_K}\to \HH^1(K,\Pic_X^0)\) and this group is torsion, \(L^{\otimes n}\) is represented by some \(M_1\in \Pic(X_{\overline{K}})^{G_K}\). Since \(X\) has index \(1\), the relative Brauer group vanishes, and hence \(\Pic(X)=\Pic(X_{\overline{K}})^{G_K}\) by \cite[Section~2]{Sti10}. Thus \(M_1\in \Pic(X)\). Now let \(M_0\in \Pic(X_{\overline{K}})\) be any lift of \(L\). Then \(M_0^{\otimes n}\otimes M_1^{-1}\in \Pic^0(X_{\overline{K}})\), so there exists \(M_2\in \Pic^0(X_{\overline{K}})\) such that
    \[M_1=(M_0\otimes M_2)^{\otimes n}.\]
    By definition, \(\delta_1 L(\sigma)=M_0^{\sigma}\otimes M_0^{-1}\). Define a cocycle \(c\in H^1(G_K,\Pic_X^0[n])\) by
    \[c(\sigma)=N^{\sigma}\otimes N^{-1},\]
    where \(N:=M_0\otimes M_2\). Note that \(N^{\sigma}\otimes N^{-1}\in \Pic_X[n]\), since \(N^{\otimes n}=M_1\in \Pic(X)\), and \(N^{\sigma}\otimes N^{-1}\in \Pic^0\), since \(N=M_0\otimes M_2\) maps to the Galois-invariant class \(L\). Moreover, under the map \(H^1(G_K,\Pic_X^0[n])\to H^1(G_K,\Pic_X[n])\), the class \(c\) maps to \(\delta M_1\), where \(\delta M_1\) is defined in \cref{thm:Chern vs Weil}, while under the map \(H^1(G_K,\Pic_X^0[n])\to H^1(G_K,\Pic_X^0)\), it maps to \(\delta_1 L\), since \(M_2\in \Pic^0(X_{\overline{K}})\). 

    Now, by the diagram in Step 1, \cref{thm:Chern vs Weil} and \cref{prop: 0 cycle orthogonal},
    \begin{equation}\label{eq: the formula}
        e(\delta_1 L\cup (y-z))=e(c\cup [s^a_{-z}-s_O])=e(c\cup [s^a-s_z^a])=e(\delta M_1\cup [s-s_z])=s^\ast c_1(M_1).
    \end{equation}    

    \textbf{Step 3. Forward direction. } Now assume \textup{(K)}. By \cref{prop: 0 cycle orthogonal}, every point of \(\Alb_X^1(K)\) represented by a \(0\)-cycle of degree \(1\) corresponds to a section orthogonal to \(\Pic\). Hence, it remains to prove the converse: if \(s:G_K\to \pi_1(X)^{(\ab)}\) is a section orthogonal to \(\Pic\), then the associated \(K\)-rational point of \(\Alb_X^1\) is represented by a \(0\)-cycle of degree \(1\).

    Under the assumption of \textup{(K)} and with the notation of Step 2, it is enough to prove that the pairing of \(y-z\) with \(\operatorname{NS}(X_{\overline{K}})^{G_K}\) is trivial. By \cref{eq: the formula},
    \[e(\delta_1L\cup(y-z))=s^\ast c_1(M_1)=0,\]
    and hence \(y\) is represented by a \(0\)-cycle of degree \(1\).

    \textbf{Step 4. Converse direction: lifting. } Now, let \(y\in \Alb_X^1(K)\) be a point such that its pairing with \(\operatorname{NS}(X_{\overline{K}})^{G_K}\) is trivial. We first prove that \(y\) is associated with a section \(s:G_K\to \pi_1(X)^{(\ab)}\). Let \(s_y^a:G_K\to \pi_1(\Alb_X^1)^{(\ab)}\) be the section corresponding to \(y\). We need to lift it to \(\pi_1(X)^{(\ab)}\). Note that, as in the discussion after \cref{prop: Brauer converse}, we have a Pontryagin duality between
    \[0\to \Pic^0(X_{\overline{K}})_{\mathrm{tors}}\to \Pic(X_{\overline{K}})_{\mathrm{tors}}\to \operatorname{NS}(X_{\overline{K}})_{\mathrm{tors}}\to 0\]
    (the second map is surjective since \(\Pic^0(X_{\overline{K}})_{\mathrm{tors}}\) is divisible) and
    \[0\to N\to \pi_1(X_{\overline{K}})^{\ab}\to \pi_1(\Alb_{X,\overline{K}}^1)^{\ab}\to 0,\]
    where \(N\) is defined as the kernel of the second map. In particular, this gives an exact sequence
    \[\HH^1(G_K,\pi_1(X_{\overline{K}})^{\ab})\to \HH^1(K,\pi_1(\Alb_{X,\overline{K}}^1)^{\ab})\xrightarrow{\rho}\HH^2(K,N).\]
    Hence, to obtain the lifting, it is enough to prove that \(\rho([s_y^a-s_z^a])=0\). By local Tate duality, this is equivalent to \(e(L\cup \rho[s_y^a-s_z^a])=0\) for every \(L\in \operatorname{NS}(X_{\overline{K}})^{G_K}_{\mathrm{tors}}\). Consider the following diagram
    \[
    \begin{tikzcd}[row sep=3em, column sep=0.55em]
    \operatorname{NS}(X_{\overline K})^{G_K}_{\mathrm{tors}} \arrow[d] & \times & \HH^2(K,N) \arrow[rr] & \phantom{\qquad} & |[yshift=-1.7em,overlay]| \Br(K) \\
    \HH^1(K,\Pic^0(X_{\overline K})_{\mathrm{tors}}) & \times & \HH^1(K,\pi_1(\Alb^1_{X,\overline K})^{\mathrm{ab}}) \arrow[u,"\rho"'] \arrow[urr] & &
    \end{tikzcd}
    \]
    We use the same notation as in Step 2. Then
    \[e(L\cup \rho[s_y^a-s_z^a])=e(\delta M_1\cup [s_y^a-s_z^a])=e(\delta_1 L\cup [y-z])=0\]
    by the assumption on \(y\). Therefore, \(\rho([s_y^a-s_z^a])=0\), and \(s_y^a\) lifts to \(s_y:G_K\to \pi_1(X)^{(\ab)}\).

    \textbf{Step 5. Converse direction: orthogonality. } We now modify an arbitrary lift by a class in \(H^1(K,N)\) so as to kill all Chern class obstructions; local Tate duality identifies the required correction as in \cref{prop: Brauer converse}. Choose a lift \(s_y\) of \(s_y^a\). We will find another lift \(s_y'\) that is orthogonal to \(\Pic\). If such a lift exists, then by the assumption, \(y\) is represented by a \(0\)-cycle of degree \(1\), and the proof is complete.

    Recall the notation \(F^1\Pic\subseteq \HH^1(K,\Pic(X_{\overline{K}})_{\mathrm{tors}})\) from the proof of \cref{prop: Brauer converse}. Again, we are in the same situation since the relative Brauer group vanishes. It is the subgroup consisting of classes \(c\) such that there exist \(n\ge 1\) and \(L\in \Pic(X)\) with \(c_{1,n}(L_{\overline{K}})=0\), and \(c\) is the image of \(\delta L\) under the map \(\HH^1(K,\Pic(X_{\overline{K}})[n])\to\HH^1(K,\Pic(X_{\overline{K}})_{\mathrm{tors}})\), where \(\delta\) is defined as in \cref{prop:chern class} as follows: choose \(M\in \Pic(X_{\overline{K}})\) such that \(M^{\otimes n}=L\). Then \(\delta L\) is defined by \(\sigma\mapsto M^{\sigma}\otimes M^{-1}\). 
    
    We will prove that \(s_y^\ast\) defines a map from \(\operatorname{Im}(F^1\Pic\to \HH^1(K,\operatorname{NS}(X_{\overline{K}})_{\mathrm{tors}}))\) to \(\Br(K)\). The only issue is well-definedness, and it is enough to prove the following: if \(\delta L\in F^1\Pic\) lies in the image of \(\HH^1(K,\Pic^0(X_{\overline{K}})_{\mathrm{tors}})\), then \(s_y^\ast c_1(L)=0\). By this assumption, there exists \(N\in \Pic(X_{\overline{K}})_{\mathrm{tors}}\) such that 
    \[(MN)^{\sigma}\otimes (MN)^{-1}\in \Pic^0(X_{\overline{K}})_{\text{tors}}.\]
    Hence, we are in the same situation as in Step 2, i.e. \(L, M_0, M_1, M_2\) of Step 2 are roughly \([M\otimes N], M\otimes N, L, \mathcal{O}_X\) here. More precisely, we need to take some care with the choice of torsion level. Let \(r\) be a common multiple of \(n\) and the order of \(N\), and write \(r=kn\). Then \(N^{\otimes r}\simeq \mathcal O\) and \((M\otimes N)^{\otimes r}\simeq L^{\otimes k}.\) Moreover, since \((M\otimes N)^\sigma\otimes(M\otimes N)^{-1}\in \Pic^0(X_{\overline K})_{\text{tors}}\), the class \([M\otimes N]\in \text{NS}(X_{\overline K})\) is \(G_K\)-invariant. Applying \cref{eq: the formula} with \([M\otimes N]\) in place of \(L\) and \(L^{\otimes k}\) in place of \(M_1\), we obtain
\[s_y^\ast c_{1,r}(L^{\otimes k})=e(\delta_1[M\otimes N]\cup(y-z))=0.\]
On the other hand, by the naturality of the Kummer sequences with respect to the inclusion \(\mu_n\hookrightarrow\mu_r\), the image of \(s_y^\ast c_{1,n}(L)\) under \(H^2(K,\mu_n)\to H^2(K,\mu_r)\) is \(s_y^\ast c_{1,r}(L^{\otimes k})\). This map is injective, since under the natural embeddings into \(\Br(K)\) it is identified with the inclusion \(\Br(K)[n]\hookrightarrow\Br(K)[r]\). Therefore,
\[s_y^\ast c_{1,n}(L)=0.\]
    Hence the map from \(\operatorname{Im}(F^1\Pic\to \HH^1(K,\operatorname{NS}(X_{\overline{K}})_{\mathrm{tors}}))\) to \(\Br(K)\) is well defined. Since \(\Br(K)\) is injective, it extends to a map from \(\HH^1(K,\operatorname{NS}(X_{\overline{K}})_{\mathrm{tors}})\), which corresponds to an element \(r\in \HH^1(K,N)\). Therefore, \(s_y'=s_y-r\) gives the desired section, by the same argument as in the proof of \cref{prop: Brauer converse}.
    
\end{proof}

\begin{rmk}
    \begin{enumerate}
        \item The projectivity hypothesis is included because it is assumed in \cite{Kai16}.
        \item One subtlety is that, even if we assume the existence of a birational geometrically abelian section, i.e. a section of \(G_{K(X)}^{(\ab)}\to G_K\), we do not know in general whether the index of \(X\) is \(1\). Over a \(p\)-adic field, \cite[Remark~(2) of Section~2.2]{BCS08} gives examples of cubic surfaces with index \(>1\) and vanishing elementary obstruction. By \cite[Theorem~3.2]{EW10}, such varieties nevertheless admit birational geometrically abelian sections. Therefore, the additional assumption that the index is \(1\) is essential. However, as we will see below, this assumption is unnecessary for curves.
    \end{enumerate}
\end{rmk}

\begin{cor}\label{cor: NS 0}
    Assume moreover that \(\operatorname{NS}(X_{\overline{K}})\) is torsion-free. Then \textup{(K)} is equivalent to the following statement: a section \(G_K\to\pi_1(X)^{(\ab)}\) is orthogonal to \(\Pic\) if and only if the associated section of \(\Alb_X^1\) defines a point represented by a \(0\)-cycle of degree \(1\).
\end{cor}

This is a direct consequence of \cref{thm:Kai and Sec} and \cref{thm:albanese} (4). Together with \cref{cor:period index} and \cref{lem:albanese lifting,lem:Picard_vaishing}, this yields the following corollary.

\begin{cor}\label{cor: Kai for curve}
    Let \(X\) be a smooth proper curve over a \(p\)-adic field \(K\) and let \(U\subseteq X\) be an open subset satisfying \textup{(VP1)}. If there exists a section \(s_U:G_K\to\pi_1(U)^{(\ab)}\), then \(X\) has index \(1\), and the induced point of \(\Alb_X^1\) is represented by a \(0\)-cycle of degree \(1\).
\end{cor}

\printbibliography

@incollection {Gro95,
    AUTHOR = {Grothendieck, Alexander},
     TITLE = {Technique de descente et th\'eor\`emes d'existence en
              g\'eom\'etrie alg\'ebrique. {VI}. {L}es sch\'emas de {P}icard:
              propri\'et\'es g\'en\'erales},
 BOOKTITLE = {S\'eminaire {B}ourbaki, {V}ol.\ 7},
     PAGES = {Exp. No. 236, 221--243},
 PUBLISHER = {Soc. Math. France, Paris},
      YEAR = {1995},
      ISBN = {2-85629-040-X},
   MRCLASS = {14C22},
  MRNUMBER = {1611207},
}

@misc{OAI26,
  author = {{OpenAI}},
  title = {{The p-adic section conjecture}},
  howpublished = {OpenAI Math Release preprint
                  \href{https://github.com/openai/math/blob/main/preprints/The-p-adic-section-conjecture-September-24-2026/main.pdf}
                  {OAI:The-p-adic-section-conjecture-September-24-2026}},
  year = {2026}
}

@article{ACV26,
  author  = {Jeffrey D. Achter and Sebastian Casalaina-Martin and Charles Vial},
  title   = {A complete answer to {Albanese} base change for incomplete varieties},
  journal = {Annales de l'Institut Fourier},
  volume  = {76},
  number  = {2},
  pages   = {789--828},
  year    = {2026},
  doi     = {10.5802/aif.3721},
}

@article{BCS08,
  author  = {Borovoi, Mikhail and Colliot-Th{\'e}l{\`e}ne, Jean-Louis and Skorobogatov, Alexei N.},
  title   = {The elementary obstruction and homogeneous spaces},
  journal = {Duke Mathematical Journal},
  volume  = {141},
  number  = {2},
  pages   = {321--364},
  year    = {2008},
  doi     = {10.1215/S0012-7094-08-14124-9}
}

@article {Don05,
    AUTHOR = {Arapura, Donu},
     TITLE = {The {L}eray spectral sequence is motivic},
   JOURNAL = {Invent. Math.},
  FJOURNAL = {Inventiones Mathematicae},
    VOLUME = {160},
      YEAR = {2005},
    NUMBER = {3},
     PAGES = {567--589},
      ISSN = {0020-9910,1432-1297},
   MRCLASS = {14F42 (14D07)},
  MRNUMBER = {2178703},
MRREVIEWER = {Byungheup\ Jun},
       DOI = {10.1007/s00222-004-0416-x},
       URL = {https://doi.org/10.1007/s00222-004-0416-x},
}

@article {DS08,
    AUTHOR = {Eriksson, Dennis and Scharaschkin, Victor},
     TITLE = {On the {B}rauer-{M}anin obstruction for zero-cycles on curves},
   JOURNAL = {Acta Arith.},
  FJOURNAL = {Acta Arithmetica},
    VOLUME = {135},
      YEAR = {2008},
    NUMBER = {2},
     PAGES = {99--110},
      ISSN = {0065-1036,1730-6264},
   MRCLASS = {14H25 (11G35 14C15 14F22 14G25)},
  MRNUMBER = {2453526},
MRREVIEWER = {Patrick\ K.\ Corn},
       DOI = {10.4064/aa135-2-1},
       URL = {https://doi.org/10.4064/aa135-2-1},
}

@book {Gir71,
    AUTHOR = {Giraud, Jean},
     TITLE = {Cohomologie non ab\'elienne},
    SERIES = {Die Grundlehren der mathematischen Wissenschaften},
    VOLUME = {Band 179},
 PUBLISHER = {Springer-Verlag, Berlin-New York},
      YEAR = {1971},
     PAGES = {ix+467},
   MRCLASS = {14F20 (14L99 18D30 18F20 55B99 55F65)},
  MRNUMBER = {344253},
MRREVIEWER = {R.\ T.\ Hoobler},
}

@article{Kai16,
  author  = {Kai, Wataru},
  title   = {A higher-dimensional generalization of Lichtenbaum duality in terms of the Albanese map},
  journal = {Compositio Mathematica},
  volume  = {152},
  number  = {9},
  pages   = {1915--1934},
  year    = {2016},
  doi     = {10.1112/S0010437X16007600}
}

@article {Lie07,
    AUTHOR = {Lieblich, Max},
     TITLE = {Moduli of twisted sheaves},
   JOURNAL = {Duke Math. J.},
  FJOURNAL = {Duke Mathematical Journal},
    VOLUME = {138},
      YEAR = {2007},
    NUMBER = {1},
     PAGES = {23--118},
      ISSN = {0012-7094,1547-7398},
   MRCLASS = {14D20 (14A20 14F05)},
  MRNUMBER = {2309155},
MRREVIEWER = {Nicolas\ Perrin},
       DOI = {10.1215/S0012-7094-07-13812-2},
       URL = {https://doi.org/10.1215/S0012-7094-07-13812-2},
}

@article {Mat55,
    AUTHOR = {Mattuck, Arthur},
     TITLE = {Abelian varieties over {$p$}-adic ground fields},
   JOURNAL = {Ann. of Math. (2)},
  FJOURNAL = {Annals of Mathematics. Second Series},
    VOLUME = {62},
      YEAR = {1955},
     PAGES = {92--119},
      ISSN = {0003-486X},
   MRCLASS = {14.0X},
  MRNUMBER = {71116},
MRREVIEWER = {C.\ Chevalley},
       DOI = {10.2307/2007101},
       URL = {https://doi.org/10.2307/2007101},
}

@book {Mil06,
    AUTHOR = {Milne, J. S.},
     TITLE = {Arithmetic duality theorems},
   EDITION = {Second},
 PUBLISHER = {BookSurge, LLC, Charleston, SC},
      YEAR = {2006},
     PAGES = {viii+339},
      ISBN = {1-4196-4274-X},
   MRCLASS = {14F20 (11G25 11G35 11R34 12G10)},
  MRNUMBER = {2261462},
}

@article {MT03,
    AUTHOR = {Spie\ss, Michael and Szamuely, Tam\'as},
     TITLE = {On the {A}lbanese map for smooth quasi-projective varieties},
   JOURNAL = {Math. Ann.},
  FJOURNAL = {Mathematische Annalen},
    VOLUME = {325},
      YEAR = {2003},
    NUMBER = {1},
     PAGES = {1--17},
      ISSN = {0025-5831,1432-1807},
   MRCLASS = {14C25 (14F42)},
  MRNUMBER = {1957261},
MRREVIEWER = {Vladimir\ I.\ Guletski\u i},
       DOI = {10.1007/s00208-002-0359-8},
       URL = {https://doi.org/10.1007/s00208-002-0359-8},
}

@article{Ser58,
     author = {Serre, Jean-Pierre},
     title = {Morphismes universels et vari\'et\'e {d'Albanese}},
     journal = {S\'eminaire Claude Chevalley},
     note = {talk:10},
     pages = {1--22},
     year = {1958-1959},
     publisher = {Secr\'etariat math\'ematique},
     volume = {4},
     zbl = {0123.13903},
     language = {fr},
     url = {https://www.numdam.org/item/SCC_1958-1959__4__A10_0/}
}

@book {Ser88,
    AUTHOR = {Serre, Jean-Pierre},
     TITLE = {Algebraic groups and class fields},
    SERIES = {Graduate Texts in Mathematics},
    VOLUME = {117},
      NOTE = {Translated from the French},
 PUBLISHER = {Springer-Verlag, New York},
      YEAR = {1988},
     PAGES = {x+207},
      ISBN = {0-387-96648-X},
   MRCLASS = {14K30 (11G45 11S31 14L40)},
  MRNUMBER = {918564},
       DOI = {10.1007/978-1-4612-1035-1},
       URL = {https://doi.org/10.1007/978-1-4612-1035-1},
}

@incollection {Tat63,
    AUTHOR = {Tate, John},
     TITLE = {Duality theorems in {G}alois cohomology over number fields},
 BOOKTITLE = {Proc. {I}nternat. {C}ongr. {M}athematicians ({S}tockholm,
              1962)},
     PAGES = {288--295},
 PUBLISHER = {Inst. Mittag-Leffler, Djursholm},
      YEAR = {1963},
   MRCLASS = {10.66},
  MRNUMBER = {175892},
}

@article {Wit08,
    AUTHOR = {Wittenberg, Olivier},
     TITLE = {On {A}lbanese torsors and the elementary obstruction},
   JOURNAL = {Math. Ann.},
  FJOURNAL = {Mathematische Annalen},
    VOLUME = {340},
      YEAR = {2008},
    NUMBER = {4},
     PAGES = {805--838},
      ISSN = {0025-5831,1432-1807},
   MRCLASS = {14C25 (14K05)},
  MRNUMBER = {2372739},
MRREVIEWER = {Francisco\ J.\ Plaza Mart\'in},
       DOI = {10.1007/s00208-007-0170-7},
       URL = {https://doi.org/10.1007/s00208-007-0170-7},
}

@article {SZ08,
    AUTHOR = {Skorobogatov, Alexei N. and Zarhin, Yuri G.},
     TITLE = {A finiteness theorem for the {B}rauer group of abelian
              varieties and {$K3$} surfaces},
   JOURNAL = {J. Algebraic Geom.},
  FJOURNAL = {Journal of Algebraic Geometry},
    VOLUME = {17},
      YEAR = {2008},
    NUMBER = {3},
     PAGES = {481--502},
      ISSN = {1056-3911,1534-7486},
   MRCLASS = {14F22 (14F20 14J28 14K15)},
  MRNUMBER = {2395136},
MRREVIEWER = {Andrei\ D.\ Halanay},
       DOI = {10.1090/S1056-3911-07-00471-7},
       URL = {https://doi.org/10.1090/S1056-3911-07-00471-7},
}

@article {BES15,
    AUTHOR = {Borne, Niels and Emsalem, Michel and Stix, Jakob},
     TITLE = {Lifting {G}alois sections along torsors},
   JOURNAL = {Math. Nachr.},
  FJOURNAL = {Mathematische Nachrichten},
    VOLUME = {288},
      YEAR = {2015},
    NUMBER = {17-18},
     PAGES = {1939--1951},
      ISSN = {0025-584X,1522-2616},
   MRCLASS = {14H30 (14G05)},
  MRNUMBER = {3434290},
MRREVIEWER = {Marco\ A.\ Garuti},
       DOI = {10.1002/mana.201400251},
       URL = {https://doi.org/10.1002/mana.201400251},
}

@article {BT08,
    AUTHOR = {Bogomolov, Fedor and Tschinkel, Yuri},
     TITLE = {Reconstruction of function fields},
   JOURNAL = {Geom. Funct. Anal.},
  FJOURNAL = {Geometric and Functional Analysis},
    VOLUME = {18},
      YEAR = {2008},
    NUMBER = {2},
     PAGES = {400--462},
      ISSN = {1016-443X,1420-8970},
   MRCLASS = {11R58 (14G32)},
  MRNUMBER = {2421544},
MRREVIEWER = {Gabriel\ D.\ Villa-Salvador},
       DOI = {10.1007/s00039-008-0665-8},
       URL = {https://doi.org/10.1007/s00039-008-0665-8},
}

@article{EW09,
  author  = {H{\'e}l{\`e}ne Esnault and Olivier Wittenberg},
  title   = {Remarks on cycle classes of sections of the arithmetic fundamental group},
  journal = {Moscow Mathematical Journal},
  volume  = {9},
  number  = {3},
  year    = {2009},
  pages   = {451--467}
}

@article {EW10,
    AUTHOR = {Esnault, H\'el\`ene and Wittenberg, Olivier},
     TITLE = {On abelian birational sections},
   JOURNAL = {J. Amer. Math. Soc.},
  FJOURNAL = {Journal of the American Mathematical Society},
    VOLUME = {23},
      YEAR = {2010},
    NUMBER = {3},
     PAGES = {713--724},
      ISSN = {0894-0347,1088-6834},
   MRCLASS = {14G32 (14C25)},
  MRNUMBER = {2629985},
MRREVIEWER = {Pierre\ A.\ Lochak},
       DOI = {10.1090/S0894-0347-10-00660-0},
       URL = {https://doi.org/10.1090/S0894-0347-10-00660-0},
}

@incollection {Gro97,
    AUTHOR = {Grothendieck, Alexandre},
     TITLE = {Esquisse d'un programme},
 BOOKTITLE = {Geometric {G}alois actions, 1},
    SERIES = {London Math. Soc. Lecture Note Ser.},
    VOLUME = {242},
     PAGES = {5--48},
      NOTE = {With an English translation on pp.\ 243--283},
 PUBLISHER = {Cambridge Univ. Press, Cambridge},
      YEAR = {1997},
      ISBN = {0-521-59642-4},
   MRCLASS = {14H10 (14F20 14H30)},
  MRNUMBER = {1483107},
MRREVIEWER = {Hiroaki\ Nakamura and Yasutaka\ Ihara},
}

@incollection {Gro97b,
    AUTHOR = {Grothendieck, Alexander},
     TITLE = {Brief an {G}. {F}altings},
 BOOKTITLE = {Geometric {G}alois actions, 1},
    SERIES = {London Math. Soc. Lecture Note Ser.},
    VOLUME = {242},
     PAGES = {49--58},
      NOTE = {With an English translation on pp.\ 285--293},
 PUBLISHER = {Cambridge Univ. Press, Cambridge},
      YEAR = {1997},
      ISBN = {0-521-59642-4},
   MRCLASS = {14E20 (14A15)},
  MRNUMBER = {1483108},
MRREVIEWER = {Helmut\ V\"olklein},
}

@article {HS09,
    AUTHOR = {Harari, David and Szamuely, Tam\'as},
     TITLE = {Galois sections for abelianized fundamental groups},
      NOTE = {With an appendix by E. V. Flynn},
   JOURNAL = {Math. Ann.},
  FJOURNAL = {Mathematische Annalen},
    VOLUME = {344},
      YEAR = {2009},
    NUMBER = {4},
     PAGES = {779--800},
      ISSN = {0025-5831,1432-1807},
   MRCLASS = {14G32 (14H25)},
  MRNUMBER = {2507624},
       DOI = {10.1007/s00208-008-0327-z},
       URL = {https://doi.org/10.1007/s00208-008-0327-z},
}

@article {Lic69,
    AUTHOR = {Lichtenbaum, Stephen},
     TITLE = {Duality theorems for curves over {$p$}-adic fields},
   JOURNAL = {Invent. Math.},
  FJOURNAL = {Inventiones Mathematicae},
    VOLUME = {7},
      YEAR = {1969},
     PAGES = {120--136},
      ISSN = {0020-9910,1432-1297},
   MRCLASS = {14.40},
  MRNUMBER = {242831},
MRREVIEWER = {Manfred\ Herrmann},
       DOI = {10.1007/BF01389795},
       URL = {https://doi.org/10.1007/BF01389795},
}

@article{Moc99,
author = {Mochizuki, Shinichi},
year = {1999},
month = {01},
pages = {319-423},
title = {The local pro-p anabelian geometry of curves},
volume = {138},
journal = {Inventiones Mathematicae},
doi = {10.1007/s002220050381}
}

@article {Moc05,
    AUTHOR = {Mochizuki, Shinichi},
     TITLE = {Galois sections in absolute anabelian geometry},
   JOURNAL = {Nagoya Math. J.},
  FJOURNAL = {Nagoya Mathematical Journal},
    VOLUME = {179},
      YEAR = {2005},
     PAGES = {17--45},
   MRCLASS = {14G32 (14G20 14G25 14H30)},
  MRNUMBER = {2164400},
MRREVIEWER = {Tam\'as\ Szamuely},
       DOI = {10.1017/S0027763000025599},
       URL = {https://doi.org/10.1017/S0027763000025599},
}

@article {Moc12,
    AUTHOR = {Mochizuki, Shinichi},
     TITLE = {Topics in absolute anabelian geometry {I}: generalities},
   JOURNAL = {J. Math. Sci. Univ. Tokyo},
  FJOURNAL = {The University of Tokyo. Journal of Mathematical Sciences},
    VOLUME = {19},
      YEAR = {2012},
    NUMBER = {2},
     PAGES = {139--242},
      ISSN = {1340-5705},
   MRCLASS = {14G32 (12F10 14L40 20E18)},
  MRNUMBER = {2987306},
MRREVIEWER = {Jon\ Eivind\ Vatne},
}

@article {Pop10c,
    AUTHOR = {Pop, Florian},
     TITLE = {On the birational {$p$}-adic section conjecture},
   JOURNAL = {Compos. Math.},
  FJOURNAL = {Compositio Mathematica},
    VOLUME = {146},
      YEAR = {2010},
    NUMBER = {3},
     PAGES = {621--637},
      ISSN = {0010-437X,1570-5846},
   MRCLASS = {14G32 (11G20 12E30 12L12)},
  MRNUMBER = {2644930},
MRREVIEWER = {Tam\'as\ Szamuely},
       DOI = {10.1112/S0010437X09004436},
       URL = {https://doi.org/10.1112/S0010437X09004436},
}

@inbook{Pop11, place={Cambridge}, series={London Mathematical Society Lecture Note Series}, title={Lectures on anabelian phenomena in geometry and arithmetic}, booktitle={Non-abelian Fundamental Groups and Iwasawa Theory}, publisher={Cambridge University Press}, author={Pop, Florian}, year={2011}, pages={1–55}, collection={London Mathematical Society Lecture Note Series}}

@article {Pop17,
    AUTHOR = {Pop, Florian},
     TITLE = {The {$\mathbb{Z}/p$} metabelian birational {$p$}-adic section
              conjecture for varieties},
   JOURNAL = {Compos. Math.},
  FJOURNAL = {Compositio Mathematica},
    VOLUME = {153},
      YEAR = {2017},
    NUMBER = {7},
     PAGES = {1433--1445},
      ISSN = {0010-437X,1570-5846},
   MRCLASS = {14G32 (11G20 12F10)},
  MRNUMBER = {3705263},
MRREVIEWER = {Ariyan\ Javanpeykar},
       DOI = {10.1112/S0010437X1700714X},
       URL = {https://doi.org/10.1112/S0010437X1700714X},
}

@unpublished{Pop25,
      title={Generalizations and minimalistic refinements of the t-birational Section Conjecture}, 
      author={Florian Pop},
      year={2026},
      eprint={2604.26131},
      archivePrefix={arXiv},
      primaryClass={math.AG},
      url={https://arxiv.org/abs/2604.26131}, 
      note={},
}

@inbook{Sai11, place={Cambridge}, series={London Mathematical Society Lecture Note Series}, title={Around the Grothendieck anabelian section conjecture}, booktitle={Non-abelian Fundamental Groups and Iwasawa Theory}, publisher={Cambridge University Press}, author={Saïdi, Mohamed}, editor={Coates, John and Kim, Minhyong and Pop, Florian and Saïdi, Mohamed and Schneider, Peter}, year={2011}, pages={72–106}, collection={London Mathematical Society Lecture Note Series}}

@article {Sai12,
    AUTHOR = {Sa\"idi, Mohamed},
     TITLE = {The cuspidalisation of sections of arithmetic fundamental
              groups},
   JOURNAL = {Adv. Math.},
  FJOURNAL = {Advances in Mathematics},
    VOLUME = {230},
      YEAR = {2012},
    NUMBER = {4-6},
     PAGES = {1931--1954},
      ISSN = {0001-8708,1090-2082},
   MRCLASS = {14G32 (11R32 12F10)},
  MRNUMBER = {2927358},
MRREVIEWER = {Hiroaki\ Nakamura},
       DOI = {10.1016/j.aim.2012.04.002},
       URL = {https://doi.org/10.1016/j.aim.2012.04.002},
}

@article {Sai21,
    AUTHOR = {Sa\"idi, Mohamed},
     TITLE = {Arithmetic of {$p$}-adic curves and sections of geometrically
              abelian fundamental groups},
   JOURNAL = {Math. Z.},
  FJOURNAL = {Mathematische Zeitschrift},
    VOLUME = {297},
      YEAR = {2021},
    NUMBER = {3-4},
     PAGES = {1191--1203},
      ISSN = {0025-5874,1432-1823},
   MRCLASS = {11G20 (14H30 14K15)},
  MRNUMBER = {4229597},
MRREVIEWER = {Christopher\ David\ Lazda},
       DOI = {10.1007/s00209-020-02553-1},
       URL = {https://doi.org/10.1007/s00209-020-02553-1},
}

@book{Ser01,
  title={Galois Cohomology},
  author={Serre, J.P.},
  isbn={9783540421924},
  lccn={97020214},
doi={10.1007/978-3-642-59141-9_1},
  series={Springer Monographs in Mathematics},
  url={https://doi.org/10.1007/978-3-642-59141-9_1},
  year={2001},
  publisher={Springer Berlin Heidelberg}
}

@article{ST22,
  author  = {Mohamed Sa{\"i}di and Akio Tamagawa},
  title   = {The $m$-step solvable anabelian geometry of number fields},
  journal = {Journal f{\"u}r die reine und angewandte Mathematik},
  volume  = {789},
  year    = {2022},
  pages   = {153--186},
  doi     = {10.1515/crelle-2022-0025}
}

@book {Sti02b,
    AUTHOR = {Stix, Jakob},
     TITLE = {Projective anabelian curves in positive characteristic and
              descent theory for log-\'etale covers},
    SERIES = {Bonner Mathematische Schriften [Bonn Mathematical
              Publications]},
    VOLUME = {354},
      NOTE = {Dissertation, Rheinische Friedrich-Wilhelms-Universit\"at
              Bonn, Bonn, 2002},
 PUBLISHER = {Universit\"at Bonn, Mathematisches Institut, Bonn},
      YEAR = {2002},
     PAGES = {xviii+118},
   MRCLASS = {14H30 (14E20 14H10)},
  MRNUMBER = {2012864},
MRREVIEWER = {Mark\ Kisin},
}

@article {Sti10,
    AUTHOR = {Stix, Jakob},
     TITLE = {On the period-index problem in light of the section conjecture},
   JOURNAL = {Amer. J. Math.},
  FJOURNAL = {American Journal of Mathematics},
    VOLUME = {132},
      YEAR = {2010},
    NUMBER = {1},
     PAGES = {157--180},
      ISSN = {0002-9327,1080-6377},
   MRCLASS = {14G32 (14F22 14G05 14H30)},
  MRNUMBER = {2597509},
MRREVIEWER = {Oscar\ G.\ Villareal},
       DOI = {10.1353/ajm.0.0092},
       URL = {https://doi.org/10.1353/ajm.0.0092},
}

@article {Sti11,
    AUTHOR = {Stix, Jakob},
     TITLE = {The {B}rauer-{M}anin obstruction for sections of the
              fundamental group},
   JOURNAL = {J. Pure Appl. Algebra},
  FJOURNAL = {Journal of Pure and Applied Algebra},
    VOLUME = {215},
      YEAR = {2011},
    NUMBER = {6},
     PAGES = {1371--1397},
      ISSN = {0022-4049,1873-1376},
   MRCLASS = {14G05 (11G30 14G25 14G32 14H30)},
  MRNUMBER = {2769238},
MRREVIEWER = {\c S.\ A.\ Basarab},
       DOI = {10.1016/j.jpaa.2010.08.017},
       URL = {https://doi.org/10.1016/j.jpaa.2010.08.017},
}

@book {Sti13b,
    AUTHOR = {Stix, Jakob},
     TITLE = {Rational points and arithmetic of fundamental groups},
    SERIES = {Lecture Notes in Mathematics},
    VOLUME = {2054},
      NOTE = {Evidence for the section conjecture},
 PUBLISHER = {Springer, Heidelberg},
      YEAR = {2013},
     PAGES = {xx+249},
      ISBN = {978-3-642-30673-0},
   MRCLASS = {14G32 (11G20 14G05)},
  MRNUMBER = {2977471},
MRREVIEWER = {Marco\ A.\ Garuti},
       DOI = {10.1007/978-3-642-30674-7},
       URL = {https://doi.org/10.1007/978-3-642-30674-7},
}

@article {Wic14,
    AUTHOR = {Wickelgren, Kirsten},
     TITLE = {2-nilpotent real section conjecture},
   JOURNAL = {Math. Ann.},
  FJOURNAL = {Mathematische Annalen},
    VOLUME = {358},
      YEAR = {2014},
    NUMBER = {1-2},
     PAGES = {361--387},
   MRCLASS = {14H30 (14P25 55P91)},
  MRNUMBER = {3158001},
MRREVIEWER = {Ambrus\ P\'al},
       DOI = {10.1007/s00208-013-0967-5},
       URL = {https://doi.org/10.1007/s00208-013-0967-5},
}

\end{document}